\documentclass[11pt,reqno]{amsart}
\usepackage{tikz}
\usepackage{subfig}
\usepackage{epstopdf}
\usepackage{epsfig}
\UseRawInputEncoding
\usepackage{math}
\graphicspath{{./figures/}}
\usepackage{tikz}
\usetikzlibrary{shapes,arrows}
\tikzstyle{decision} = [diamond, draw, fill=blue!20, 
    text width=4.5em, text badly centered, node distance=3cm, inner sep=0pt]
\tikzstyle{block} = [rectangle, draw, fill=blue!20, 
    text width=5em, text centered, rounded corners, minimum height=4em]
\tikzstyle{line} = [draw, -latex']
\tikzstyle{cloud} = [draw, ellipse,fill=red!20, node distance=3cm,
    minimum height=2em]

\newcommand{\ts}{\mathsf{T}}
\newcommand{\hh}{\mathrm{H}}
\usetikzlibrary{positioning}
\tikzset{main node/.style={circle,fill=blue!20,draw,minimum size=1cm,inner sep=0pt},  }

\begin{document}
\title[Transport Hessian metric]{Hessian metric via transport information geometry}
\author[Li]{Wuchen Li}
\email{wuchen@mailbox.sc.edu}
\address{Mathematics department, University of South Carolina, Columbia, 29036.}
\thanks{W. Li is supported by a start up funding in University of South Carolina.}

\begin{abstract}
We propose to study the Hessian metric of {a functional on the space of probability measures endowed with the Wasserstein $2$-metric}. We name it transport Hessian metric, which contains and extends the classical Wasserstein-$2$ metric. We formulate several dynamical systems associated with transport Hessian metrics. Several connections between transport Hessian metrics and mathematical physics equations are discovered. For examples, the transport Hessian gradient flow, including Newton's flow, formulates a mean-field kernel Stein variational gradient flow; The transport Hessian Hamiltonian flow of Boltzmann-Shannon entropy forms the Shallow water equation; The transport Hessian gradient flow of Fisher information {leads to} the heat equation. {Several examples and closed-form solutions for transport Hessian distances are presented.}
\end{abstract}
\maketitle
\section{Introduction}
{Metric spaces of probability measures} {\cite{AGS1, AGS}} play crucial roles in differential geometry \cite{LiG, Villani2009_optimal}, mathematical physics \cite{Carlen, li_SE, LC, Modin, Lafferty, Nelson0, Nelson}, and information theory \cite{CoverThomas1991_elements}. One typical application of the metric is the Bayesian sampling problem \cite{sampling}, which is to sample a target distribution based on a prior distribution. In this case, the metric can be used to design objective functions. Famous examples include Kullback--Leibler (KL) divergence or {Wasserstein-$2$} distance. Besides, the metric can be applied to design sampling directions. Typical examples include Fisher--Rao gradient, {Wasserstein-$2$} gradient and Stein gradient. These gradient flows in the probability space represent various Markov processes on the sample space, such as birth-death dynamics \cite{IG, IG2}, Langevin dynamics \cite{JKO}, and Stein variational gradient flows \cite{SGD, Liu2017}. 

There are naturally geometric variational structures, as well as mathematical physics equations, associated with these metrics. On the one hand, information geometry studies the Hessian metric in probability density space \cite{IG, IG2}. In this area, the Fisher--Rao metric, i.e., the Hessian operator of negative Boltzmann--Shannon entropy in {$L^2$ space}, plays fundamental roles. It is an invariant metric to derive divergence functions, including the KL divergence. On the other hand, the Wasserstein-$2$ metric { a.k.a. transport metric}\footnote{There are many historic names for optimal transport metric \cite{Villani2009_optimal}. For the simplicity of presentation, we call Wasserstein-$2$ metric the {\em transport metric}.} has been studied in the area of optimal transport; see {\cite{AGS1, AGS, Lafferty,Villani2009_optimal}} and many references therein. Also, the transport metric deeply interacts with mathematical physics equations. A celebrated result is that the transport gradient {descent} flow of negative Boltzmann--Shannon entropy is the heat equation \cite{otto2001, JKO}. And transport Hamiltonian flows contain compressible Euler equation, Schr{\"o}dinger equation, and Schr{\"o}dinger bridge problem \cite{LC, CP, Modin, LiG}.  The Stein metrics have been studied independently \cite{Liu2017}, which formulate transport metrics with given kernel functions. The Stein metric recently find vast connections with mathematical physics equations and sampling problems \cite{SN, SGD, Liu2017}. 

To advance current works, a joint study between information geometry and optimal transport becomes essential. Here the Hessian metrics on density space are of great importance, which arises in differential geometry \cite{CY, P}, learning optimization problems \cite{SN, YL1} and variational formulations of mathematical physics equations \cite{SW}. Natural questions arise: {\em What is the Hessian metric in probability space embedded with the transport metric? What is the formulation of the transport Hessian metric related dynamics, such as gradient flows and Hamiltonian flows? Does the transport Hessian dynamics connect with mathematical physics equations?} 

In this paper, following studies in \cite{LiG}, we positively answer the above questions. We study the formulation of the Hessian metric for an energy in probability space embedded with optimal transport metric. We name such metric {\em transport Hessian metric}.  We derive both gradient and Hamiltonian flows for the transport Hessian metric. Several observations are made as follows:
\begin{itemize}
\item[(i)] The transport Hessian metric is a Stein metric with a mean field kernel. Hence the transport Newton's flow is a particular mean field Stein variational gradient flow;
\item[(ii)] At least in one dimensional sample space, the Hamiltonian flow of transport Hessian metric for negative Boltzmann-Shannon entropy forms the Shallow water equation;
\item[(iii)] We establish a relation among Boltzmann--Shannon entropy, Fisher information and transport Hessian metric. For example, the transport Hessian gradient flow of Fisher information satisfies the heat equation. 
\end{itemize}
Our result can be summarized by a variational problem as follows. Given a compact, smooth sample space $M$, denote the smooth probability density space on $M$ by $\mathcal{P}$ and a smooth energy function $\mathcal{E}\colon \mathcal{P}\rightarrow \mathbb{R}$. Consider
\begin{equation*}
\begin{split}
\inf_{\rho,\Phi}~~&\int_0^1\Big\{\int \int \nabla_x\nabla_y\delta^2\mathcal{E}(\rho)(x,y)\nabla\Phi(t,x)\nabla\Phi(t,y)\rho(t,x)\rho(t,y)dxdy\\
&\qquad+\int\nabla_x^2\delta\mathcal{E}(\rho)(x)(\nabla\Phi(t,x), \nabla\Phi(t,x))\rho(t,x)dx\Big\} dt,
\end{split}
\end{equation*}
where $\delta$, $\delta^2$ are $L^2$ first, second variation operators respectively,  and the infimum is taken among all possible density paths and potential functions $(\rho, \Phi)\colon [0,1]\times M\rightarrow\mathbb{R}^2$, such that the continuity equation with gradient drift vector fields holds
\begin{equation*}
\partial_t\rho+\nabla\cdot(\rho\nabla\Phi)=0,\qquad \textrm{fixed $\rho^0$, $\rho^1$.}
\end{equation*}  
We notice that if $\mathcal{E}(\rho)=\frac{1}{2}\int x^2\rho dx$, then {the proposed variational formulation recovers the Benamou-Brenier formula \cite{BB}, so as the Wasserstein-$2$ metric \cite{Villani2009_optimal}}. In other words, the above formulation forms 
\begin{equation*}
\begin{split}
\inf_{\rho,\Phi}~~\Big\{\int_0^1\int\|\nabla\Phi(t,x)\|^2\rho(t,x)dxdt\colon\partial_t\rho+\nabla\cdot(\rho\nabla\Phi)=0,~\textrm{fixed $\rho^0$, $\rho^1$}\Big\}. 
\end{split}
\end{equation*}
In this sense, the transport Hessian metric contains and extends the Wasserstein-$2$ metric. 

Here the result belongs to the study of transport information geometry (TIG) founded in \cite{LiG}. TIG is a newly emerged area that applies optimal transport, information geometry, and differential geometry methods to study mathematical physics equations' variational formulations. Nowadays, it has vast applications in constructing fluid dynamics, proving functional inequalities, and designing algorithms; see \cite{CL, QL, LL, LiG1, LLW, LiG2, LiG3, LiG4}. In this paper, we extend the area of TIG into the category of Hessian geometry. One direct application is to formulate optimization techniques for Bayesian sampling problems \cite{YL, YL1}. See related developments in information geometry \cite{MP1, MP2}. In the current paper, we discuss the connections between transport Hessian metrics and mathematical physics equations. We also demonstrate the relation between transport Hessian gradient flows and Stein variational sampling algorithms. 

{Compared to our previous works, we observe that the Wasserstein-$2$ metric is a Hessian metric of second moment functional in the transport density manifold. We generalize this observation to general Hessian metrics in Wasserstein-$2$ space and derive Hamiltonian flows therein.}

This paper is arranged as follows. Section \ref{section2} establishes formulations of transport Hessian metrics and associated distance functions.  
We derive transport Hessian gradient flows in section \ref{section4}, whose connections with Stein gradient flows are shown in section \ref{section6}. In particular, we show that a particular mean field kernel version of Stein variational gradient flow {leads to} the transport Newton's flow. 
We present Hessian Hamiltonian flows in section \ref{section5} and point out the connection between Boltzmann--Shanon entropy and Shallow water equations. {Several closed-form solutions of transport Hessian distances are presented in section \ref{section7}.} 

\section{Optimal transport Hessian metric}\label{section2}
In this section, we briefly review the definition of the Hessian metric on a finite-dimensional {sample space. We} recall the formulation of the optimal transport metric defined in probability density space. Using these concepts, we introduce the Hessian metric of {an} energy in probability density space with Wasserstein-$2$ metric. Several examples are provided. 
\subsection{Finite dimensional Hessian metric}
Let $(M, g)$ be a smooth, compact, $d$-dimensional Riemannian manifold without boundaries, where $g$ is a metric tensor of $M$. For concreteness, let $M=\mathbb{T}^{d}$, where $\mathbb{T}^d$ represents a $d$ dimensional torus. 
In this case, denote {by} $\nabla$ the Euclidean gradient operator. Here the tangent space and the cotangent space at $x\in M$ form 
\begin{equation*}
T_xM=\big\{\dot x\in \mathbb{R}^d\big\}, \quad T_x^*M=\big\{p=g(x)^{-1}\dot x\in \mathbb{R}^d\big\}.
\end{equation*}
We consider a strictly convex energy function $E\colon M\rightarrow \mathbb{R}$. We will study the Hessian metric of $E$ in $(M,g)$. There are several equivalent formulations of Hessian operators defined on $(M, g)$.  On the one hand, the Hessian operator of $E$ in $(M, g)$ can be defined on the tangent space, i.e. 
\begin{equation*}
\textrm{Hess}_gE\colon M\times T_xM\times T_xM\rightarrow \mathbb{R},\quad \textrm{Hess}_gE=\Big((\textrm{Hess}_gE)_{ij}\Big)_{1\leq i,j\leq d}\in \mathbb{R}^{d\times d}.
\end{equation*}
Here 
\begin{equation*}
(\textrm{Hess}_gE(x))_{ij}=(\nabla^2E(x))_{ij}-\sum_{k=1}^d\nabla_{x_k}E(x)\Gamma^k_{ij}(x),
\end{equation*}
where $\Gamma^k_{ij}\colon M\rightarrow\mathbb{R}$ is the Christoffel symbol defined by 
\begin{equation*}
\Gamma^k_{ij}(x)=\frac{1}{2}\sum_{k'=1}^d(g(x)^{-1})_{kk'}\Big(\nabla_{x_i}g_{jk'}(x)+\nabla_{x_j}g_{ik'}(x)-\nabla_{x_{k'}}g_{ij}(x)\Big).
\end{equation*}
We note that in our example, $\Gamma^k_{ij}$ represents the torsion free Levi-Civita connection, i.e. $\Gamma_{ij}^k=\Gamma_{ji}^k$. On the other hand, the Hessian operator can be formulated as a bilinear form on the cotangent space. Denote 
$$\textrm{Hess}^*_gE\colon T^*_xM\times T^*_xM\rightarrow \mathbb{R},\quad \textrm{Hess}^*_gE=\Big((\textrm{Hess}_g^*E)_{ij}\Big)_{1\leq i,j\leq d}\in \mathbb{R}^{d\times d},$$ then 
\begin{equation*}
\dot x^{\ts}\textrm{Hess}_gE(x)\dot x=p^{\ts}\textrm{Hess}^*_gE(x) p,
\end{equation*}
where $p=g(x)^{-1}\dot x$, for any $\dot x\in T_xM$. In other words, we have
\begin{equation*}
\textrm{Hess}_g^*E(x)=g(x)^{-1}\textrm{Hess}_gE(x)g(x)^{-1}.
\end{equation*}
Hence the Hessian metric of an energy function $E$ in $(M,g)$ is defined as follows:
\begin{equation*}
 g^\hh(x)=\textrm{Hess}_gE(x)=g(x) \textrm{Hess}_g^*E(x) g(x). 
\end{equation*}
{From now on, we assume $g^{\hh}$ is semi-positive definite. Here $(M,  g^\hh)$ is named Hessian manifold}. Later on, we will use the above two formulations of the Hessian metric, depending on which is more convenient.   
\subsection{Optimal transport metric}
We next present the Wasserstein-$2$ metric and demonstrate its associated Hessian operator of a functional. 

Let $(M, g)=(\mathbb{T}^d, \mathbb{I})$ be a $d$ dimensional torus, where $\mathbb{I}\in \mathbb{R}^{d\times d}$ is an identity matrix and $(\cdot, \cdot)$ denotes the Euclidean inner product. Denote a smooth positive probability density space by 
\begin{equation*}
\mathcal{P}=\Big\{\rho\in C^{\infty}(M)\colon \int \rho dx=1,\quad \rho> 0\Big\}.
\end{equation*}
The tangent space at $\rho \in \mathcal{P}$ is given by 
\begin{equation*}
T_\rho \mathcal{P}=\Big\{\sigma\in C^\infty(M)\colon \int \sigma dx=0\Big\}.
\end{equation*}

To define the optimal transport metric in the probability space, we need the following notation convention. Define a weighted Laplacian operator by 
\begin{equation*}
\Delta_a=\nabla\cdot(a\nabla), 
\end{equation*}
where $a\in C^{\infty}(M)$ is a given smooth function. Here for any testing functions $f_1$, $f_2\in C^{\infty}(M)$, we have
\begin{equation*}
\int (f_1, \Delta_a f_2) dx= -\int (\nabla f_1, \nabla f_2)a dx. 
\end{equation*}
Later on, we will use this weighted Laplacian operator to define both the metric and the Levi-Civita connection in probability density space. Here for the metric operator, $a$ is chosen as a non-negative probability density function. For the Levi-Civita connection, $a$ is selected as a tangent vector in probability density space, which is not a non-negative function.  

We are ready to present the transport metric. 
\begin{definition}[Transport metric]\label{metric}
  The inner product $\g\colon \mathcal{P}\times {T_\rho}\mathcal{P}\times{T_\rho}\mathcal{P}\rightarrow\mathbb{R}$ is defined by  
  \begin{equation*} 
 \g(\rho)(\sigma_1, \sigma_2)=\int \big(\sigma_1, (-\Delta_\rho)^{-1}\sigma_2\big) dx,
  \end{equation*}
  for any $\sigma_1,\sigma_2\in T_\rho\mathcal{P}$. Here $\Delta_\rho=-\nabla\cdot(\rho\nabla)$ is an elliptic operator weighted linearly by density function $\rho$. On the other hand, denote $\sigma_i=-\Delta_\rho \Phi_i=-\nabla\cdot(\rho\nabla\Phi_i)$, $i=1,2$, then 
  \begin{equation*}
  \begin{split}
   \g(\rho)(\sigma_1, \sigma_2)=&\int \big(\Phi_1, (-\Delta_\rho)(-\Delta_\rho)^{-1}(-\Delta_\rho)\Phi_2\big) dx\\
=&\int \big(\Phi_1, -\nabla\cdot(\rho\nabla\Phi_2)\big) dx\\
=&\int (\nabla\Phi_1, \nabla\Phi_2)\rho dx. 
\end{split}
  \end{equation*} 
  \end{definition}
  In this case, $(\mathcal{P}, \g)$ forms an infinite-dimensional Riemannian manifold, named transport density manifold \cite{Lafferty}. We comment that the optimal transport metric induces a distance function, which has other formulations, such as a linear programming problem with a given ground cost function or a mapping formulation, named Monge problems \cite{Villani2009_optimal}. Throughout this paper, we only use the metric formulation of optimal transport formulated in Definition \ref{metric}.   
  
We next introduce the Hessian operator in transport density manifold. Denote {by} $\delta$, $\delta^2$ the $L^2$ first and second variation operators, respectively. Here the cotangent space at $\rho\in\mathcal{P}$ satisfies
\begin{equation*}
 T_\rho^*\mathcal{P}=\Big\{\Phi=(-\Delta_\rho)^{-1}\sigma\in C^{\infty}(M)\colon \textrm{$\Phi$ is uniquely determined by a constant shrift}\Big\}.
\end{equation*}
Given a smooth functional $\mathcal{E}\colon \mathcal{P}\rightarrow \mathbb{R}$, the Hessian operator of $\mathcal{E}$ in $(\mathcal{P}, \g)$ has the following two formulations. 
On the one hand, the Hessian operator of $\mathcal{E}$ in $(\mathcal{P}, \g)$ can be defined on the tangent space, i.e. 
$$\textrm{Hess}_\g\mathcal{E}\colon \mathcal{P}\times T_\rho\mathcal{P}\times T_\rho\mathcal{P}\rightarrow \mathbb{R},$$ 
which satisfies 
\begin{equation}\label{Hess}
\begin{split}
\textrm{Hess}_\g\mathcal{E}(\rho)(\sigma_1, \sigma_2)=&\int \int \delta^2\mathcal{E}(\rho)\sigma_1(x)\sigma_2(y)dxdy-\int \delta\mathcal{E}(\rho)(x)\mathbf{\Gamma}(\rho)(\sigma_1,\sigma_2)(x) dx,
\end{split}
\end{equation}
where $\mathbf{\Gamma}\colon \mathcal{P}\times T_\rho\mathcal{P}\times T_\rho\mathcal{P}\rightarrow\mathbb{R}$ is the Christoffel symbol in $(\mathcal{P}, \g)$ defined by 
\begin{equation}\label{WC}
\begin{split}
{\bf\Gamma}(\rho)(\sigma_1,\sigma_2)=&-\frac{1}{2}\Big\{\Delta_{\sigma_1}\Delta_{\rho}^{-1}\sigma_2+\Delta_{\sigma_2}\Delta_{\rho}^{-1}\sigma_1+\Delta_\rho(\nabla\Delta_\rho^{-1}\sigma_1, \nabla\Delta_\rho^{-1}\sigma_2)  \Big\}.
\end{split}
\end{equation}
We note that $\mathbf{\Gamma}$ represents the torsion free Levi-Civita connection in transport density manifold, i.e. $$\mathbf{\Gamma}(\rho)(\sigma_1, \sigma_2)(x)=\mathbf{\Gamma}(\rho)(\sigma_2,\sigma_1)(x).$$ On the other hand, the Hessian operator can be formulated as a bilinear form on the cotangent space. Denote 
\begin{equation*}
\textrm{Hess}^*_\g \mathcal{E}\colon \mathcal{P}\times T^*_\rho\mathcal{P}\times T^*_\rho\mathcal{P}\rightarrow \mathbb{R},
\end{equation*}
and
\begin{equation*}
\sigma_i=-\nabla\cdot(\rho\nabla\Phi_i),\quad i=1,2. 
\end{equation*}
Then by direct calculations, we have
\begin{equation}\label{formula}
\begin{split}
\textrm{Hess}_\g^*\mathcal{E}(\rho)(\Phi_1, \Phi_2)=&\int \int \delta^2\mathcal{E}(\rho)(x,y)(-\Delta_\rho\Phi_1)(x)(-\Delta_\rho\Phi_2)(y)dxdy\\
&-\int \delta\mathcal{E}(\rho)(x)\mathbf{\Gamma}(\rho)(-\Delta_\rho\Phi_1, -\Delta_\rho\Phi_2)(x)dx\\
=&\int \int \nabla_x\nabla_y\delta^2\mathcal{E}(\rho)(x,y)\nabla\Phi_1(x)\nabla\Phi_2(y)\rho(x)\rho(y)dxdy\\
&+\int\nabla_x^2\delta\mathcal{E}(\rho)(\nabla\Phi_1(x), \nabla\Phi_2(x))\rho(x)dx.
\end{split}
\end{equation}
Here the second equality holds by integration by parts formula. In particular, the {direct calculation} of Christoffel symbol \eqref{WC} in transport density manifold \cite{LiG} shows that 
\begin{equation*}
\begin{split}
\int \delta\mathcal{E}(\rho)(x)\mathbf{\Gamma}(\rho)(\sigma_1,\sigma_2)(x) dx=&-\frac{1}{2}\int \delta\mathcal{E}(\rho)(x)\Big\{\nabla\cdot(\nabla\cdot(\rho\nabla\Phi_1)\nabla\Phi_2)+\nabla\cdot(\nabla\cdot(\rho\nabla\Phi_2)\nabla\Phi_1)\\
&\hspace{2.7cm} +\nabla\cdot(\rho\nabla(\nabla\Phi_1,\nabla\Phi_2))\Big\} dx\\
=&\frac{1}{2}\int\Big\{-\big(\nabla(\nabla\delta\mathcal{E}, \nabla\Phi_1), \nabla\Phi_2\big)-\big(\nabla(\nabla\delta\mathcal{E}, \nabla\Phi_2), \nabla\Phi_1\big)\\
&\hspace{2.7cm} +\big(\nabla \delta\mathcal{E}(\rho), \nabla(\nabla\Phi_1,\nabla\Phi_2)\big)\Big\} \rho(x)dx\\
=&-\int \nabla_x^2\delta\mathcal{E}(\rho)(\nabla\Phi_1(x),\nabla\Phi_2(x))\rho(x)dx.
\end{split}
\end{equation*}
Similar to the finite dimensional case, the Hessian operator in {transport} density manifold has the following relation:
\begin{equation*}
\textrm{Hess}_\g\mathcal{E}(\rho)=\Delta_\rho^{-1}\cdot \textrm{Hess}_{\g}^*\mathcal{E}(\rho)\cdot \Delta_\rho^{-1},
\end{equation*}
 where the operator $\cdot$ is in the sense of $L^2$ inner product. 

\subsection{Optimal transport Hessian metric}
We are now ready to present the transport Hessian metric. 
\begin{definition}[Transport Hessian metric]
  The inner product $ \g^\hh(\rho)\colon
  \mathcal{P}\times {T_\rho}\mathcal{P}\times{T_\rho}\mathcal{P}\rightarrow\mathbb{R}$ is defined by  
  \begin{equation*} 
  \g^\hh(\rho)(\sigma_1, \sigma_2)=\textrm{Hess}_\g\mathcal{E}(\rho)(\sigma_1,\sigma_2),
  \end{equation*}
  for any $\sigma_1,\sigma_2\in T_\rho\mathcal{P}$. Here $\textrm{Hess}_\g$ is defined by \eqref{Hess}. In details, 
  \begin{equation*}
  \begin{split}
     \g^\hh(\rho)(\sigma_1, \sigma_2)=& \int \int \nabla_x\nabla_y\delta^2\mathcal{E}(\rho)(x,y)(\nabla\Phi_1(x), \nabla\Phi_2(y))\rho(x)\rho(y)dxdy\\
&+\int\nabla_x^2\delta\mathcal{E}(\rho)(\nabla\Phi_1(x), \nabla\Phi_2(x))\rho(x)dx,
  \end{split}
  \end{equation*} 
  where 
  \begin{equation*}
    \sigma_i=-\nabla\cdot(\rho\nabla\Phi_i), \qquad i=1,2. 
  \end{equation*}  
\end{definition}
Here $ \g^\hh_\rho$ is the Hessian metric of energy $\mathcal{E}$ in optimal transport. For this reason, we call $(\mathcal{P},  \g^\hh(\rho))$ the {\em Hessian density manifold}. 
From now on, we only consider the case that $\ g^\hh(\rho)$ is a positive definite operator for all $\rho\in \mathcal{P}$. In other words, we assume that $\mathcal{E}(\rho)$ is geodesically convex in $(\mathcal{P}, \g)$. {These functionals $\mathcal{E}$ are well studied in optimal transport theory, namely displacement convex functionals. See details in \cite[Chapter 9]{AGS}. } 

We next introduce that the transport Hessian metric of {a displacement convex functional} induces a distance function, $\mathrm{Dist}_{\mathrm{H}}\colon \mathcal{P}\times \mathcal{P}\rightarrow \mathbb{R}$. 
Here $\mathrm{Dist}_{\mathrm{H}}$ can be given by the following action functional in $(\mathcal{P},  \g^\hh(\rho))$:
\begin{equation*}
\mathrm{Dist}_{\mathrm{H}}(\rho^0,\rho^1)^2:=\inf_{\rho\colon [0,1]\rightarrow\mathcal{P}}~~\Big\{\int_0^1 \g^\hh(\rho)(\partial_t\rho, \partial_t\rho) dt\colon~~\textrm{fixed $\rho^0$, $\rho^1$}\Big\},
\end{equation*}
where the infimum is taken among all density paths $\rho\colon [0,1]\times M\rightarrow \mathbb{R}$. In other words, we arrive at the following definition of distance function.  
\begin{definition}[Transport Hessian distance]
\begin{equation}\label{p1}
\begin{split}
\mathrm{Dist}_{\mathrm{H}}(\rho^0,\rho^1)^2=\inf_{\rho,\Phi}~~&\int_0^1\Big\{\int \int \nabla_x\nabla_y\delta^2\mathcal{E}(\rho)(x,y)\nabla\Phi(t,x)\nabla\Phi(t,y)\rho(t,x)\rho(t,y)dxdy\\
&\qquad+\int\nabla_x^2\delta\mathcal{E}(\rho)(\nabla\Phi(t,x), \nabla\Phi(t,x))\rho(t,x)dx\Big\} dt,
\end{split}
\end{equation}
where the infimum is taken among all possible density and potential functions $(\rho, \Phi)\colon [0,1]\times M\rightarrow\mathbb{R}^2$, such that the continuity equation with gradient drift vector field holds
\begin{equation*}
\partial_t\rho+\nabla\cdot(\rho\nabla\Phi)=0,\quad \textrm{fixed $\rho^0$, $\rho^1$.}
\end{equation*}
\end{definition}

{\begin{remark}[Connections with Wasserstein-2 distance]
We notice that if $\mathcal{E}(\rho)=\int \frac{x^2}{2}\rho(x)dx$, then variational formulation \eqref{p1} leads to the Benamou-Breiner formula \cite{BB}, which is the dynamical formulation of Wasserstein-2 distance \cite[Chapter 8]{AGS}. I.e. 
\begin{equation*}
\begin{split}
\mathrm{Dist}_{\mathrm{H}}(\rho^0,\rho^1)^2=\inf_{\rho,\Phi}~~&\int_0^1\int \|\nabla\Phi\|^2\rho dx dt,
\end{split}
\end{equation*}
where the infimum is taken among the following constraint 
\begin{equation*}
\partial_t\rho+\nabla\cdot(\rho\nabla\Phi)=0,\quad \textrm{fixed $\rho^0$, $\rho^1$.}
\end{equation*}
In this case, $\mathrm{Dist}_{\mathrm{H}}$ recovers the classical Wasserstein-2 distance. 
\end{remark}}
{\begin{remark}[Connections with Lagrangian coordinates]
The proposed Hessian transport metrics have the Lagrangian formulations as in \cite{AGS}. We leave detailed studies in the future work. 
\end{remark}}

\subsection{Examples}
We present several examples of transport Hessian metrics. For the simplicity of presentation, several transport Hessian metrics are given directly; see their derivations in \cite{Villani2009_optimal} and many references therein. 
In later on examples, we always denote 
\begin{equation*}
\Phi_i\in T_\rho^*\mathcal{P}\quad\textrm{s.t.}\quad \sigma_i=-\nabla\cdot(\rho\nabla\Phi_i)\in T_\rho\mathcal{P},\quad i=1,2.
\end{equation*}
\begin{example}[Linear energy]
Consider 
\begin{equation*}
\mathcal{E}(\rho)=\int E(x)\rho(x)dx, 
\end{equation*}
where $E\in C^{\infty}(M)$ is a given strictly convex potential function. Then $\delta\mathcal{E}(\rho)(x)=E(x)$ and $\delta^2\mathcal{E}(\rho)(x,y)=0$. Hence \begin{equation*}
 \g^\hh(\rho)(\sigma_1, \sigma_2)=\int \nabla^2E(x)(\nabla\Phi_1(x),\nabla\Phi_2(x))\rho(x)dx. 
\end{equation*}
In particular, if $E(x)=\frac{x^2}{2}$, then
\begin{equation*}
 \g^\hh(\rho)(\sigma_1, \sigma_2)=\int (\nabla\Phi_1(x),\nabla\Phi_2(x))\rho(x)dx.
\end{equation*}
In this case, our transport Hessian metric is exactly the Wasserstein-$2$ metric; see \cite{Villani2009_optimal}. Hence the transport Hessian metric contains and extends the formulation of the optimal transport metric. 
\end{example}
\begin{example}[Interaction energy]
Consider 
\begin{equation*}
\mathcal{E}(\rho)=\frac{1}{2}\int \int W(x,y)\rho(x)\rho(y)dxdy, 
\end{equation*}
where $W\in C^{\infty}(M\times M)$ is a given kernel potential function. Then $\delta\mathcal{E}(\rho)(x)=\int W(x,y)\rho(y)dy$ and $\delta^2\mathcal{E}(\rho)(x,y)=W(x,y)$. Hence 
\begin{equation*}
\begin{split}
 \g^\hh(\rho)(\sigma_1, \sigma_2)=&\int \int \Big[\nabla_{xy}^2W(x,y)(\nabla_x\Phi_1(x),\nabla_y\Phi_2(y))\\
&\qquad+\nabla_x^2W(x,y)(\nabla\Phi_1(x), \nabla\Phi_2(x))\Big]\rho(x)\rho(y)dxdy.
\end{split}
\end{equation*}
A concrete examples is given as follows: If $W(x,y)=\frac{\|x-y\|^2}{2}$, then $-\nabla_{xy}^2W(x,y)=\nabla_x^2W(x,y)=\mathbb{I}$. Hence
\begin{equation*}
\begin{split}
 \g^\hh(\rho)(\sigma_1, \sigma_2)=&\int (\nabla\Phi_1(x), \nabla\Phi_2(x))\rho dx-\Big(\int\nabla\Phi_1(y)\rho(y)dy, \int\nabla\Phi_2(y)\rho(y)dy\Big) \\
=&\int \Big(\nabla\Phi_1(x)-\int \nabla\Phi_1(y)\rho(y)dy,   \nabla\Phi_2(x)-\int \nabla\Phi_2(y)\rho(y)dy\Big)\rho(x)dx.
\end{split}
\end{equation*}
\end{example}

\begin{example}[Entropy]
Consider 
\begin{equation*}
\mathcal{E}(\rho)=\int f(\rho)(x) dx,
\end{equation*}
where $f\in C^{\infty}(\mathbb{R})$ is a given strict convex function. In this case, $\delta_\rho\mathcal{E}(\rho)(x)=f'(\rho)(x)$ and $\delta^2\mathcal{E}(\rho)(x,y)=f''(\rho)\delta(x-y)$, where $\delta(x-y)$ is a delta function. Hence 
\begin{equation}\label{ent}
\begin{split}
 \g^\hh(\rho)(\sigma_1, \sigma_2)=&\int \Big\{f''(\rho)(x)\nabla\cdot(\rho(x)\nabla\Phi_1(x))\nabla\cdot(\rho(x)\nabla\Phi_2(x))\\
&\qquad+ \nabla^2f'(\rho)(x)(\nabla\Phi_1(x), \nabla\Phi_2(x))\rho(x)\Big\}dx\\
=&\int \mathrm{tr}(\nabla^2\Phi_1(x), \nabla^2\Phi_2(x))p(\rho)(x)+(\Delta\Phi_1(x), \Delta\Phi_2(x))p_2(\rho)(x)dx,
\end{split}
\end{equation}
where $\mathrm{tr}$ denotes the matrix trace operator, and functions $p$, $p_2\colon \mathbb{R}\rightarrow \mathbb{R}$ satisfy
\begin{equation*}
p(\rho)=\rho f'(\rho)-f(\rho),\quad p_2(\rho)=\rho p'(\rho)-p(\rho).
\end{equation*}
Several concrete examples are provided as follows:
\begin{itemize}
\item[(i)] If $f(\rho)=\rho\log\rho$, then $-\int \rho\log\rho dx$ is known as the {Boltzmann--Shannon entropy}. Here $f'(\rho)=\log\rho+1$, $f''(\rho)=\frac{1}{\rho}$, hence $p(\rho)=\rho$, $p_2(\rho)=0$. Then 
\begin{equation}\label{SE}
 \g^\hh(\rho)(\sigma_1, \sigma_2)=\int \mathrm{tr}(\nabla^2\Phi_1(x),\nabla^2\Phi_2(x)) \rho(x) dx.
\end{equation}
\item[(ii)] If $f(\rho)=\frac{1}{2}\rho^2$, then $\mathcal{E}(\rho)=\frac{1}{2}\int\rho^2 dx$. Here $f'(\rho)=\rho$, $f''(\rho)=1$. Hence $p(\rho)=p_2(\rho)=\frac{1}{2}\rho^2$. Then
\begin{equation*}
 \g^\hh(\rho)(\sigma_1, \sigma_2)=\frac{1}{2}\int \Big(\mathrm{tr}(\nabla^2\Phi_1(x), \nabla^2\Phi_2(x))+(\Delta\Phi_1(x), \Delta\Phi_2(x))\Big)\rho(x)^2dx.
\end{equation*}
\end{itemize}
\end{example}
\begin{remark}
If $(M, g)$ is a Riemannian manifold, then Hessian operator in \eqref{ent} also contains the Ricci curvature tensor on $(M, g)$. For the simplicity of presentation, we assume that manifold $(M, g)$ is Ricci flat, i.e. $\textrm{Ric}_M=0$.  
\end{remark}
\begin{remark}
We notice that the second equation in formula \eqref{ent} has been formulated in \cite{Villani2009_optimal}; where the first equality equals to the second equality can be shown by using Wasserstein Christoffel symbol \eqref{WC}; see \cite{LiG, LiG1}. This fact relates to the construction of Bakry--{\'E}mery Gamma calculus; see related works in \cite{BE, QL, LiG1}. 
\end{remark}
\begin{remark}
We remark that in one dimensional sample space, the transport Hessian metric defined in \eqref{SE} coincides with the semi--invariant metric studied in \cite{Modin1}. 
\end{remark}
\section{Transport Hessian gradient flows}\label{section4}
In this section, we derive gradient flows in the Hessian density manifold and provide the associated entropy dissipation properties. Several examples are given. 

We first derive the gradient flow in Hessian density manifold $(\mathcal{P},  \g^\hh)$. From now on, we consider a smooth energy $\mathcal{F}\colon \mathcal{P}\rightarrow\mathbb{R}$.
\begin{theorem}[Transport Hessian gradient flow]\label{thm4}
The gradient flow of $\mathcal{F}(\rho)$ in $(\mathcal{P},  \g^\hh)$ satisfies 
\begin{equation}\label{tgd}
\partial_t\rho(t,x)=\nabla\cdot(\rho(t,x)\nabla\Phi^{\mathcal{F}}(x,\rho)),
\end{equation}
where $\Phi^{\mathcal{F}}\in C^\infty(M\times \mathcal{P})$ satisfies the following Poisson equation: 
\begin{equation}\label{PO}
\begin{split}
&\nabla_x\cdot(\rho(t,x) \nabla_x\int \nabla_y\delta^2\mathcal{E}(\rho)(x,y)\nabla_y{\Phi^\mathcal{F}}(y,\rho)\rho(t,y)dy)+\nabla_x\cdot(\rho(t,x)\nabla_x^2\delta\mathcal{E}(\rho)(x)\nabla_x\Phi^\mathcal{F}(x,\rho))\\
=&\nabla_x\cdot(\rho(t,x)\nabla_x\delta\mathcal{F}(\rho)(x)).
\end{split}
\end{equation}
\end{theorem}
\begin{proof}
The proof follows the definition of the gradient operator.

\noindent\textbf{Claim:} The gradient operator of $\mathcal{F}(\rho)$ in $(\mathcal{P},  \g^\hh)$, i.e. $\textrm{grad}_{ \g^\hh}\colon \mathcal{P}\times C^\infty(\mathcal{P})\rightarrow T_\rho\mathcal{P}$, is defined by
\begin{equation*}
\textrm{grad}_{ \g^\hh}\mathcal{F}(\rho)(x)=-\nabla\cdot(\rho(x)\nabla{\Phi^\mathcal{F}}(x,\rho)),
\end{equation*}
where $\Phi^\mathcal{F}\in C^\infty(M)$ solves the Poisson equation \eqref{PO}. 

Suppose the claim is true, then the gradient flow follows \begin{equation*}
\partial_t\rho=-\textrm{grad}_{ \g^\hh}\mathcal{F}(\rho),
\end{equation*}
which finishes the proof. Now, we only need to prove the claim as follows. 
\begin{proof}[Proof of Claim]
For any $\sigma(x)\in T_\rho\mathcal{P}$, we have
\begin{equation*}
 \g^\hh(\rho)(\sigma, \textrm{grad}_{ \g^\hh}\mathcal{F}(\rho))=\int \delta\mathcal{F}(\rho)(x)\sigma(x)dx.
\end{equation*}
Denote $\Phi,~\Phi^\mathcal{F}\in C^{\infty}(M\times \mathcal{P})$, such that 
\begin{equation*}
\sigma=-\nabla\cdot(\rho\nabla\Phi), \quad\textrm{and}\quad \textrm{grad}_{ \g^\hh}\mathcal{F}(\rho)=-\nabla\cdot(\rho\nabla\Phi^\mathcal{F}).
\end{equation*}
Then from \eqref{formula}, we obtain
\begin{equation*}
\begin{split}
 \g^\hh(\rho)(\sigma, \textrm{grad}_{ \g^\hh}\mathcal{F}(\rho))=&\textrm{Hess}_\g\mathcal{E}(\rho)(-\nabla\cdot(\rho\nabla\Phi), -\nabla\cdot(\rho\nabla\Phi^\mathcal{F}))\\
=&\textrm{Hess}_\g^*\mathcal{E}(\rho)(\Phi, \Phi^{\mathcal{F}})\\
=&\int \int \nabla_x\nabla_y\delta^2\mathcal{E}(\rho)(x,y)\nabla_x\Phi(x)\nabla_y\Phi^\mathcal{F}(y,\rho)\rho(x)\rho(y)dxdy\\
&+\int\nabla_x^2\delta\mathcal{E}(\rho)(\nabla\Phi(x), \nabla\Phi^\mathcal{F}(x,\rho))\rho(x)dx\\
=&-\int\nabla_x\cdot\Big(\rho(x)\nabla_x\int \nabla_y\delta^2\mathcal{E}(\rho)(x,y)\nabla\Phi^\mathcal{F}(y,\rho)\rho(y)dy\Big) \Phi(x)dx\\
&-\int \nabla_x\cdot(\rho\nabla_x^2\delta\mathcal{E}(\rho)\nabla\Phi^{\mathcal{F}}(x,\rho))\Phi(x)dx,
\end{split}
\end{equation*}
where the last equality follows from the integration by parts formula. Here 
\begin{equation*}
\begin{split}
\int \delta\mathcal{F}(\rho)(x)\sigma(x)dx=&\int \delta\mathcal{F}(\rho)(x)\Big(-\nabla\cdot(\rho\nabla\Phi)(x)\Big)dx\\
=&-\int \Big(\nabla\cdot(\rho\nabla\delta\mathcal{F}(\rho)(x))\Big)\Phi(x)dx,
\end{split}
\end{equation*}
where the last equality holds by the integration by parts formula twice. We notice that the above two formulas equal to each other for any $\sigma\in T_\rho\mathcal{P}$, i.e. for any $\Phi\in C^{\infty}(M)$. This derives the Poisson equation \eqref{PO}.  
\end{proof}
\end{proof}

We next present the following two categories of gradient flows in Hessian density manifold. Firstly, we introduce a class of transport Newton's flows \cite{YL}.  
\begin{corollary}[Transport Newton's flow]
If $\mathcal{E}(\rho)=\mathcal{F}(\rho)$, then the gradient flow of $\mathcal{F}(\rho)$ in Hessian density manifold $(\mathcal{P},  \g^\hh)$ forms
\begin{equation}\label{wnew}
\left\{
\begin{split}
&\partial_t\rho(t,x)=\nabla\cdot(\rho(t,x)\nabla\Phi(x,\rho))\\
&\nabla_x\cdot(\rho(t,x) \nabla_x\int \nabla_y\delta^2\mathcal{E}(\rho)(x,y)\nabla_y\Phi(y,\rho)\rho(t,y)dy)+\nabla_x\cdot(\rho(t,x)\nabla_x^2\delta\mathcal{E}(\rho)(x)\nabla_x\Phi(x,\rho))\\
&=\nabla\cdot(\rho(t,x)\nabla\delta\mathcal{E}(\rho)(x)).
\end{split}\right.
\end{equation}
This is the Newton's flow of $\mathcal{E}(\rho)$ in density manifold $(\mathcal{P}, \g)$.
\end{corollary}
\begin{remark}
We comment that the Newton's flow in transport density manifold and the Newton's flow in $L^2$ space behave similarly in the asymptotic sense. In other words, we observe that when the density is at the minimizer, i.e. $\rho=\rho^*$, then $\nabla_x\delta\mathcal{E}(\rho^*)=0$. Here the transport Newton's direction forms the $L^2$ Newton's direction, i.e.
\begin{equation*}
\textrm{grad}_{ \g^\hh}\mathcal{E}(\rho)|_{\rho=\rho^*}=\delta^2\mathcal{E}(\rho)^{-1}\delta\mathcal{E}(\rho)|_{\rho=\rho^*}, 
\end{equation*}
This fact demonstrates that the transport Newton's direction is asymptotically a $L^2$--Newton's direction in density space. See related convergence proof of transport Newton's method in \cite{YL1}.
\end{remark}
\begin{proof}
The proof follows from the definition of Newton's direction in a manifold. Notice that the Newton's flow forms 
\begin{equation*}
\begin{split}
\partial_t\rho=&-\Big(\textrm{Hess}_{\g}\mathcal{E}(\rho)\Big)^{-1}\textrm{grad}_{\g}\mathcal{E}(\rho)\\
=&-\textrm{grad}_{ \g^\hh}\mathcal{E}(\rho). 
\end{split}
\end{equation*}
Substituting $\mathcal{F}(\rho)=\mathcal{E}(\rho)$ into equation \eqref{tgd}, we obtain the result. 
\end{proof}

Secondly, we propose a new class of metrics for deriving the classical transport gradient flows. 
\begin{corollary}[Transport gradient flow]
Consider an energy
\begin{equation*}
\mathcal{F}(\rho)=\frac{1}{2}\int \|\nabla_x\delta\mathcal{E}(\rho)(x)\|^2\rho(x)dx.
\end{equation*}
 Then the gradient flow of energy $\mathcal{F}(\rho)$ in Hessian density manifold $(\mathcal{P},  \g^\hh)$ formulates 
\begin{equation*}
\partial_t\rho(t,x)=\nabla_x\cdot(\rho(t,x) \nabla_x\delta\mathcal{E}(\rho)(x)),
\end{equation*}
which is the gradient flow of energy $\mathcal{E}(\rho)$ in density manifold $(\mathcal{P}, \g)$.
\end{corollary}
\begin{proof}
We notice that 
\begin{equation*}
\begin{split}
\mathcal{F}(\rho)=&\frac{1}{2}\int \Big(\delta\mathcal{E}(\rho), -\Delta_\rho\delta\mathcal{E}(\rho)\Big) dx \\
=&\frac{1}{2}\int \|\nabla_x\delta\mathcal{E}(\rho)(x)\|^2\rho(x) dx \\
=&\frac{1}{2} \g(\rho)(\textrm{grad}_\g\mathcal{E}(\rho), \textrm{grad}_\g\mathcal{E}(\rho)).
\end{split}
\end{equation*}
Then 
\begin{equation*}
\begin{split}
\textrm{grad}_\g\mathcal{F}(\rho)=&\textrm{grad}_\g\Big\{\frac{1}{2} \g(\rho)(\textrm{grad}_\g\mathcal{E}(\rho), \textrm{grad}_\g\mathcal{E}(\rho))\Big\}\\
=&\textrm{Hess}_g\mathcal{E}(\rho)\cdot\textrm{grad}_\g\mathcal{E}(\rho).
\end{split}
\end{equation*}
Hence the gradient flow equation of $\mathcal{F}(\rho)$ in $(\mathcal{P},  \g^\hh)$ satisfies 
\begin{equation*}
\begin{split}
\partial_t\rho=&-\textrm{grad}_{ \g^\hh}\mathcal{F}(\rho)\\
=&-\Big(\textrm{Hess}_{\g}\mathcal{E}(\rho)\Big)^{-1}\textrm{grad}_{\g}\mathcal{F}(\rho)\\
=&-\Big(\textrm{Hess}_{\g}\mathcal{E}(\rho)\Big)^{-1}\cdot \textrm{Hess}_g\mathcal{E}(\rho)\cdot \textrm{grad}_\g\mathcal{E}(\rho)\\
=&-\textrm{grad}_{\g}\mathcal{E}(\rho)\\
=&\nabla\cdot(\rho\nabla\delta\mathcal{E}(\rho)),
\end{split}
\end{equation*}
which finishes the proof.
\end{proof}

\subsection{Examples}
In this section, we present several examples of gradient flows in Hessian density manifold, and point out their connections with mathematical physics equations.  

We first present several examples of transport Newton's flows, in which we consider the energy functional $\mathcal{F}(\rho)=\mathcal{E}(\rho)$.
\begin{example}[Transport Newton's flow of linear energy]
 Consider $$\mathcal{E}(\rho)=\int E(x)\rho(x)dx,$$ then the transport Newton's flow forms 
\begin{equation*}
\left\{\begin{split}
&\partial_t\rho(t,x)=\nabla\cdot(\rho(t,x)\nabla\Phi^{\mathcal{E}}(x,\rho))\\
&\nabla\cdot(\rho(t,x)\nabla^2E(x)\nabla\Phi^\mathcal{E}(x,\rho))=\nabla\cdot(\rho(t,x)\nabla E(x)).
\end{split}\right.
\end{equation*}
\end{example}
\begin{example}[Transport Newton's flow of interaction energy]
Consider $$\mathcal{E}(\rho)=\frac{1}{2}\int\int W(x,y)\rho(x)\rho(y)dxdy,$$ 
then the transport Newton's flow satisfies 
\begin{equation*}
\left\{\begin{split}
&\partial_t\rho(t,x)=\nabla\cdot(\rho(t,x)\nabla\Phi^{\mathcal{E}}(x,\rho))\\
&\nabla\cdot\Big(\rho(t,x) \big[\int \nabla^2_{xy}W(x,y)\nabla\Phi^{\mathcal{E}}(y,\rho)\rho(t,y)dy+\int\nabla^2_xW(x,y)\nabla\Phi^{\mathcal{E}}(x,\rho)\rho(t,y)dy\big]\Big)\\
&=\nabla\cdot(\rho(t,x)\nabla\int W(x,y)\rho(t,y)dy).
\end{split}\right.
\end{equation*}
A concrete example of interaction kernel is given as follows. 
If $W(x,y)=\frac{\|x-y\|^2}{2}$, then we obtain 
\begin{equation*}
\left\{\begin{split}
&\partial_t\rho(t,x)=\nabla\cdot(\rho(t,x)\nabla\Phi^{\mathcal{E}}(x,\rho))\\
&\nabla\cdot\Big(\rho(t,x)\big[\nabla_x\Phi^{\mathcal{E}}(x,\rho)-\int\nabla_y\Phi^{\mathcal{E}}(y,\rho)\rho(t,y)dy\big]\Big)=\nabla\cdot(\rho(t,x)\nabla\int\frac{1}{2}\|x-y\|^2\rho(t,y)dy).
\end{split}\right.
\end{equation*}
\end{example}
\begin{example}[Transport Newton's flow of entropy]
Consider 
$$\mathcal{E}(\rho)=\int f(\rho)(x)dx,$$ 
then the transport Newton's flow satisfies
\begin{equation}\label{entropy}
\left\{\begin{split}
&\partial_t\rho(t,x)=\nabla\cdot(\rho(t,x)\nabla\Phi^{\mathcal{E}}(x,\rho))\\
&-\Big\{\nabla^2\colon(p(\rho(t,x))\nabla^2\Phi^{\mathcal{E}}(x,\rho))+\Delta(p_2(\rho(t,x))\Delta \Phi^{\mathcal{E}}(x,\rho))\Big\}=\nabla\cdot(\rho(t,x)\nabla f'(\rho)(x)).
\end{split}\right.
\end{equation}
Here for any function $a\in C^{\infty}(M)$, $\nabla^2\colon(a\nabla^2)$ represents the second order weighted Laplacian operator. In other words, for any testing function $f_1$, $f_2\in C^{\infty}(M)$, we have
\begin{equation*}
\int f_1\nabla^2\colon(a \nabla^2 f_2)dx=\int \mathrm{tr}(\nabla^2f_1,\nabla^2f_2)a dx.
\end{equation*}
Several concrete examples of equation \eqref{entropy} are given as follows.
\begin{itemize}
\item[(i)] If $f(\rho)=\rho\log\rho$, then we have
\begin{equation*}
\left\{\begin{split}
&\partial_t\rho(t,x)=\nabla\cdot(\rho(t,x)\nabla\Phi^{\mathcal{E}}(x,\rho))\\
&-\nabla^2\colon(\rho(t,x)\nabla^2\Phi^{\mathcal{E}}(x,\rho))=\nabla\cdot(\rho(t,x)\nabla\log\rho(t,x))=\Delta\rho(t,x).
\end{split}\right.
\end{equation*}
\item[(ii)] If $f(\rho)=\frac{\rho^2}{2}$, then we obtain 
\begin{equation*}
\left\{\begin{split}
&\partial_t\rho(t,x)=\nabla\cdot(\rho(t,x)\nabla\Phi^{\mathcal{E}}(x,\rho))\\
&-\frac{1}{2}\nabla^2\colon(\rho(t,x)^2\nabla^2\Phi^{\mathcal{E}}(x,\rho))-\frac{1}{2}\Delta(\rho(t,x)^2\Delta \Phi^{\mathcal{E}}(x,\rho))\\
&=\nabla\cdot(\rho(t,x)\nabla \rho(t,x))=\frac{1}{2}\Delta\rho^2(t,x). 
\end{split}\right.
\end{equation*}
\end{itemize}
We notice that equation \eqref{entropy} introduces transport Newton's equations for general entropies. Notice that the transport gradient flows of entropies include the heat equation, the Porous media equation \cite{Villani2009_optimal}, etc. Following this relation, we named the derived equations as {\em Newton's heat equation, Newton's Porous media equation}, etc.
\end{example}

We next present the other connection between transport Hessian metrics and heat equations.  
\begin{example}[Connections with heat equations]
Denote $\mathcal{E}(\rho)$ by the negative Boltzmann--Shannon entropy 
\begin{equation*}
\mathcal{E}(\rho)=\int \rho(x)\log\rho(x)dx.
\end{equation*}
Denote $\mathcal{F}(\rho)$ by the $1/2$ Fisher information functional 
\begin{equation*}
\mathcal{F}(\rho)=\frac{1}{2}\g(\rho)(\textrm{grad}_\g\mathcal{E}(\rho),\textrm{grad}_\g\mathcal{E}(\rho))=\frac{1}{2}\int \|\nabla\log\rho\|^2\rho dx. 
\end{equation*} 
The transport Hessian gradient flow of $1/2$ Fisher information functional forms the heat equation. 
\begin{equation*}
\begin{split}
\partial_t\rho=&-\textrm{grad}_{ \g^\hh}\mathcal{F}(\rho)\\
=&-\nabla\cdot(\rho\nabla)(\nabla^2\colon \rho \nabla^2)^{-1}\nabla\cdot(\rho\nabla)\delta_\rho\mathcal{F}(\rho)\\
=&\nabla\cdot(\rho \nabla)(\nabla^2\colon \rho \nabla^2)^{-1}(-\nabla^2\colon \rho\nabla^2)\delta\mathcal{E}(\rho)\\
=&\nabla\cdot(\rho\nabla\delta\mathcal{E}(\rho))=\nabla\cdot(\rho\nabla\log\rho)\\
=&\Delta\rho.
\end{split}
\end{equation*}
In above, we use the fact \cite{QD} that the gradient operator of Fisher information functional in $(\mathcal{P}, \g)$ satisfies 
\begin{equation*}
\textrm{grad}_{\g}\mathcal{F}(\rho)=\textrm{Hess}_\g\mathcal{E}(\rho)\cdot\textrm{grad}_{\g}\mathcal{E}(\rho)=-\nabla^2\colon(\rho\nabla^2\log\rho).
\end{equation*}
\end{example}
\begin{remark}
We notice that the gradient flow of Fisher information in Wasserstein-$2$ metric is known as the quantum heat equation; see \cite{QD} and many references therein. We summarize that the relation among heat equation, quantum heat equation and Newton's heat equation is as follows. Denote $\mathcal{E}(\rho)=\int \rho\log\rho dx$ and $\mathcal{F}(\rho)=\frac{1}{2}\int \|\nabla\log\rho\|^2\rho dx$, then
 \begin{equation*}
\begin{split}
&\partial_t\rho=-\textrm{grad}_\g\mathcal{F}(\rho)=-\textrm{Hess}_\g\mathcal{E}(\rho)\cdot\textrm{grad}_{\g}\mathcal{E}(\rho); \qquad\qquad{\textrm{Quantum heat equation}}\\
&\partial_t\rho=-\frac{1}{2}\textrm{Hess}_\g\mathcal{E}(\rho)^{-1}\cdot\textrm{grad}_\g\mathcal{F}(\rho)=-\textrm{grad}_{\g}\mathcal{E}(\rho);  \hspace{1.2cm} \textrm{Heat equation}\\
&\partial_t\rho=-\textrm{Hess}_\g\mathcal{E}(\rho)^{-1}\cdot\textrm{grad}_{\g}\mathcal{E}(\rho);\hspace{3.8cm}\textrm{Newton's heat equation}
\end{split}.
\end{equation*}
\end{remark}

\subsection{Transport Hessian entropy dissipation}
In this subsection, we demonstrate the dissipation properties of energy functionals along transport Hessian gradient flows.  
\begin{corollary}[Transport Hessian de Brun identity]
Suppose $\rho(t,x)$ satisfies the transport Hessian gradient flow \eqref{tgd}, then 
\begin{equation*}
\begin{split}
\frac{d}{dt}\mathcal{F}(\rho(t,\cdot))=&-\mathcal{I}_{H}(\rho(t,\cdot)),
\end{split}
\end{equation*}
{where $\mathcal{I}_{H}\colon \mathcal{P}\rightarrow \mathbb{R}$ is defined by
\begin{equation*}
\begin{split}
\mathcal{I}_{H}(\rho)=&\int (\nabla_x\Phi^{\mathcal{F}}(x,\rho), \nabla_x\delta\mathcal{F}(\rho)(x))\rho(x)dx,
\end{split}
\end{equation*} }
with $\Phi^\mathcal{F}$ satisfying the Poisson equation \eqref{PO}.  
\end{corollary}
\begin{proof}
The proof follows from the definition of gradient flow. In other words, 
\begin{equation*}
\begin{split}
\frac{d}{dt}\mathcal{F}(\rho)=&\int \delta\mathcal{F}(\rho)(x) \nabla_x\cdot(\rho(t,x)\nabla_x\Phi^{\mathcal{F}}(x,\rho)) dx\\
=&-\int (\nabla_x\delta\mathcal{F}(\rho)(x), \nabla_x\Phi^\mathcal{F}(x,\rho))\rho(t,x)dx,
\end{split}
\end{equation*}
which finishes the proof.
\end{proof}
We provide several examples of the proposed entropy dissipation relations. 
\begin{example}
We notice that if $\mathcal{E}(\rho)=\frac{1}{2}\int x^2\rho(x)dx$ and $\mathcal{F}(\rho)=\int \rho(x)\log\rho(x)dx$, then 
\begin{equation*}
\nabla\cdot(\rho\nabla \Phi^\mathcal{F}(x,\rho))=\nabla\cdot(\rho\nabla\log\rho), 
\end{equation*}
 i.e. 
 \begin{equation*}
  \Phi^\mathcal{F}(x,\rho)=\log\rho.
\end{equation*}
 Hence 
 \begin{equation*}
 \mathcal{I}_{H}(\rho)=\int \|\nabla\log\rho(x)\|^2\rho(x)dx,
 \end{equation*} 
 which is known as the Fisher information functional. 
Hence, for general choices of $\mathcal{E}$, functional $\mathcal{I}_{H}$ extends the definition of classical Fisher information functional. For this reason, we call $\mathcal{I}_{H}$ the {\em transport Hessian Fisher information functional}. 
\end{example}
\begin{example}
If $\mathcal{E}(\rho)=\int \rho(x)\log\rho(x) dx$ and $\mathcal{F}(\rho)=\frac{1}{2}\int \|\nabla\log\rho(x)\|^2\rho(x) dx$, then
\begin{equation*}
\nabla^2\colon(\rho\nabla^2\Phi^\mathcal{F}(x,\rho))=-\nabla\cdot(\rho\nabla\delta\mathcal{F}(\rho))=\nabla^2\colon(\rho\nabla^2\log\rho), 
\end{equation*}
i.e.
\begin{equation*}
 \Phi^\mathcal{F}(x,\rho)=\log\rho.
\end{equation*}
Hence
\begin{equation*}
\mathcal{I}_H(\rho)=\int \mathrm{tr}(\nabla^2\log\rho(x), \nabla^2\log\rho(x))\rho(x) dx.
\end{equation*}
In this case, our transport Hessian Fisher information functional recovers the second order entropy dissipation property. In other words, the dissipation of classical Fisher information functional $\int \|\nabla\log\rho\|^2\rho dx$ along heat equation equals to second order information functional $\mathcal{I}_H=\int \mathrm{tr}(\nabla^2\log\rho, \nabla^2\log\rho)\rho dx$; see \cite{P}. 
\end{example}
We notice that this entropy dissipation relation will be useful in proving related functional inequalities. We leave the related studies of functional inequalities in the future {work}. 
\section{Connections with Stein variational gradient flows}\label{section6}
In this section, we connect transport Hessian gradient flows with Stein variational gradient flows. 

To do so, we first introduce the definition of kernel functions. 
 \begin{definition}
 Denote the kernel functions $\mathrm{H}$, $K_\mathrm{H}\colon M\times M\times \mathcal{P}\rightarrow \mathbb{R}$ as follows:
\begin{equation}\label{ker}
\mathrm{H}(x,y,\rho)=\delta^2\mathcal{E}(\rho)(x,y)\nabla_x\cdot(\rho\nabla_x)\nabla_y\cdot(\rho\nabla_y)-\nabla_x\cdot(\rho \nabla_x^2\delta\mathcal{E}(\rho)(x)\nabla_x)\delta(x-y).
\end{equation}
Denote 
\begin{equation*}
K_\mathrm{H}(x,y,\rho):=\mathrm{H}(x,y,\rho)^{-1}, 
\end{equation*}
in the sense that  
\begin{equation*}
\int\int K_\mathrm{H}(x,y,\rho) \mathrm{H}(y,z,\rho) f(z) dxdy= f(z), 
\end{equation*}
for any testing function $f\in C^{\infty}(M)$.
\end{definition} 
Using the kernel function $K_\hh$, we formulate the transport Hessian metric as follows. 
\begin{theorem}\label{prop9}
The transport Hessian metric $ \g^\hh$ satisfies 
\begin{equation}\label{HM}
 \g^\hh(\rho)(\sigma_1,\sigma_2)=\int \int \nabla_x\nabla_yK_\mathrm{H}(x,y,\rho)\big(\nabla_x\Psi_1(x),\nabla_y\Psi_2(y)\big)\rho(x)\rho(y)dxdy,
\end{equation}
where {$\Psi_i\in C^{\infty}(M)$ and $\sigma_i\in T_\rho\mathcal{P}$, $i=1,2$, satisfy}
\begin{equation}\label{sigma}
\sigma_i(x)=-\nabla_x\cdot\Big(\rho(x) \int \nabla_x\nabla_yK_\mathrm{H}(x,y,\rho)\nabla_y\Psi_i(y)\rho(y)dy\Big).
\end{equation}
\end{theorem}
\begin{remark}
We remark that the transport Hessian metric is a particular Stein variational metric \cite{SGD, Liu2017} with a mean field kernel function. To see this connection, let us first recall the definition of a Stein metric. Given a kernel matrix function $K_S(x,y)=K_S(y,x)\in \mathbb{R}^{d\times d}$, for any $x,y\in M$ with suitable conditions, the Stein metric $\g^S\colon \mathcal{P}\times T_\rho\mathcal{P}\times T_\rho\mathcal{P}\rightarrow\mathbb{R}$ is defined as follows:
\begin{equation}\label{stein}
\g^S(\rho)(\sigma_1,\sigma_2)=\int\int K_S(x,y)(\nabla_x\Psi(x), \nabla_y\Psi(y))\rho(x)\rho(y)dxdy,
\end{equation}
with 
\begin{equation*}
\sigma_i(x)=-\nabla_x\cdot(\rho(x)\int \rho(y) K_S(x,y)\nabla_y\Psi_i(y) dy), \quad \textrm{for $i=1,2$.}
\end{equation*}
Comparing the definition of Stein metric \eqref{stein} with transport Hessian metric \eqref{HM}, we notice that 
\begin{equation*}
K_S(x,y)=\nabla_x\nabla_yK_\mathrm{H}(x,y,\rho).
\end{equation*}
Here the matrix kernel function $K_S$ in transport Hessian metric depends on the density function $\rho$. 
\end{remark}
\begin{proof}
We notice that
\begin{equation*}
\begin{split}
 \g^\hh(\rho)(\sigma, \sigma)=&\textrm{Hess}^*_\g\mathcal{E}(\rho)(\Phi,\Phi)\\
=&\int\int \mathrm{H}(x,y,\rho)\Phi(x)\Phi(y)dxdy,
\end{split}
\end{equation*}
where $$\sigma=-\nabla\cdot(\rho\nabla\Phi).$$ 
Here $\mathrm{H}\colon M\times M\times\mathcal{P}\rightarrow \mathbb{R}$ is the kernel function satisfying \eqref{ker}. 
Denote $\Phi=(-\Delta_\rho)^{-1}\sigma$, then  
\begin{equation*}
\begin{split}
 \g^\hh(\rho)(\sigma, \sigma)=&\int\int \mathrm{H}(x,y,\rho)(-\Delta_\rho)^{-1}\sigma(x) (-\Delta_\rho)^{-1}\sigma(y)dxdy\\
=& \Big(\sigma, (-\Delta_\rho)^{-1}\cdot \mathrm{H}\cdot (-\Delta_\rho)^{-1}\sigma\Big)_{L^2}.
\end{split}
\end{equation*}
We next represent the metric $\g^{\H}$ in the cotangent space of $(\mathcal{P},  \g^\hh)$. Denote 
\begin{equation*}
\begin{split}
\sigma =& \Big((-\Delta_\rho)^{-1}\cdot \mathrm{H}\cdot (-\Delta_\rho)^{-1}\Big)^{-1}\Psi\\
=&\Delta_\rho\cdot \mathrm{H}^{-1} \cdot \Delta_\rho \Psi \\
=&\Delta_\rho\cdot K_\mathrm{H} \cdot \Delta_\rho\Psi.
\end{split}.
\end{equation*}
By integration by parts formula, we derive equation \eqref{sigma}. Denote the operator $G_\mathrm{H}\colon T_\rho\mathcal{P}\rightarrow T_\rho\mathcal{P}$ by 
\begin{equation*}
G_{\mathrm{H}}(\rho)=(\Delta_\rho)^{-1}\cdot \mathrm{H} \cdot (\Delta_\rho)^{-1},
\end{equation*}
then
\begin{equation}
G_\hh(\rho)^{-1}=(\Delta_\rho)\cdot K_\mathrm{H}\cdot (\Delta_\rho).
\end{equation}
Hence 
\begin{equation*}
\begin{split}
\ g^\hh(\rho)(\sigma, \sigma)=&\Big( \Psi, G_\hh(\rho)^{-1} G_\hh(\rho) G_\hh(\rho)^{-1} \Psi\Big)_{L^2}\\
=&\Big( \Psi, G_\hh(\rho)^{-1} \Psi\Big)_{L^2}\\
=&\Big( \Psi, \Delta_\rho\cdot K_\mathrm{H}\cdot \Delta_\rho \Psi\Big)_{L^2}\\
=& \int \int \Psi(x)\nabla_x\cdot\Big(\rho(x)\nabla_x K_\mathrm{H}(x,y,\rho) \nabla_y\cdot(\rho(y)\nabla_y\Psi(y))\Big)dx dy\\
=& -\int \int (\nabla_x\Psi(x), \nabla_x K_\mathrm{H}(x,y,\rho)) \rho(x)\nabla_y\cdot(\rho(y)\nabla_y\Psi(y))dx dy\\
=& \int \int \nabla_x\nabla_yK_\mathrm{H}(x,y,\rho)(\nabla_x\Psi(x),\nabla_y\Psi(y))\rho(x)\rho(y)dx dy,
\end{split}
\end{equation*}
where the last two equalities apply the integration by parts w.r.t. variables $x$, $y$, respectively. 
\end{proof} 
We are now ready to formulate transport Newton's flows in term of Stein variational gradient flows.
\begin{proposition}[Stein--transport Newton's flows]
The transport Newton's flow of energy $\mathcal{E}(\rho)$ satisfies 
\begin{equation*}
\partial_t\rho(t,x)=\nabla_x\cdot\Big(\rho(t,x)\int \rho(t,y)(\nabla_x\nabla_yK_\mathrm{H}(x,y,\rho), \nabla_y\delta\mathcal{E}(\rho)(y))dy\Big).
\end{equation*}
In particular, consider $\mathcal{E}(\rho)$ as the KL divergence functional, i.e. $\mathcal{E}(\rho)=\int\rho\log\frac{\rho}{e^{-f}}dx$, then the transport Newton's flow satisfies
\begin{equation}\label{Sflow}
\begin{split}
\partial_t\rho=&\quad \nabla_x\cdot\Big(\rho(t,x)\int \rho(t,y)(\nabla_x\nabla_yK_\mathrm{H}(x,y,\rho), \nabla_y f(y)dy\Big)\\
&-\nabla_x\cdot\Big(\rho(t,x) \int \rho(t,y)\nabla_x\Delta_y K_\mathrm{H}(x,y,\rho)dy\Big).
\end{split}
\end{equation}

\end{proposition}
\begin{proof}
From Theorem \ref{thm4},  the transport Hessian gradient flow satisfies the following equation
\begin{equation}\label{wnew1}
\left\{
\begin{split}
&\partial_t\rho-\nabla\cdot(\rho\nabla\Phi^{\xi})=0\\
&-\int \mathrm{H}(x,y,\rho)\Phi^{\xi}(y,\rho)dy=\nabla\cdot(\rho\nabla\delta\mathcal{E}(\rho))(x).
\end{split}\right.
\end{equation}
Denote $\Phi=-\Phi^\xi$. We notice that the second equation of \eqref{wnew1} satisfies
\begin{equation*}
\begin{split}
\Phi(x,\rho)=&\int K_\mathrm{H}(x,y,\rho) \nabla_y\cdot(\rho(y)\nabla_y\delta\mathcal{E}(\rho))dy\\
=&-\int (\nabla_yK_\mathrm{H}(x,y,\rho), \nabla_y\delta\mathcal{E}(\rho))\rho(y)dy.
\end{split}
\end{equation*}
Hence the first equation of \eqref{wnew1} satisfies 
\begin{equation*}
\begin{split}
\partial_t\rho=&-\nabla\cdot(\rho\nabla\Phi)\\
=&\nabla_x\cdot\Big(\rho\nabla_x \int (\nabla_yK_\mathrm{H}(x,y,\rho), \nabla_y\delta\mathcal{E}(\rho)(y))\rho(t,y)dy \Big)\\
=&\nabla_x\cdot\Big(\rho\int \rho(t,y)(\nabla_x\nabla_yK_\mathrm{H}(x,y,\rho), \nabla_y\delta\mathcal{E}(\rho)(y))dy\Big),
\end{split}
\end{equation*}
which finishes the proof of {the} first part. {Denote by $\mathcal{E}(\rho)$ the KL divergence}. Then $\delta\mathcal{E}(\rho)=\log\frac{\rho}{e^{-f}}+1$. Hence 
\begin{equation*}
\begin{split}
\partial_t\rho=&\quad \nabla_x\cdot\Big(\rho(t,x)\int \rho(t,y)(\nabla_x\nabla_yK_\mathrm{H}(x,y,\rho), \nabla_y\log\frac{\rho(t,y)}{e^{-f(y)}}dy\Big)\\
=&\quad\nabla_x\cdot\Big(\rho(t,x)\int \rho(t,y)(\nabla_x\nabla_yK_\mathrm{H}(x,y,\rho), \nabla_y\log\rho(t,y))dy\Big)\\
&+\nabla_x\cdot\Big(\rho(t,x)\int \rho(t,y)(\nabla_x\nabla_yK_\mathrm{H}(x,y,\rho), \nabla_yf(y))dy\Big)\\
=&\quad\nabla_x\cdot\Big(\rho(t,x)\int (\nabla_x\nabla_yK_\mathrm{H}(x,y,\rho), \nabla_y\rho(t,y))dy\Big)\\
&+\nabla_x\cdot\Big(\rho(t,x)\int \rho(t,y)(\nabla_x\nabla_yK_\mathrm{H}(x,y,\rho), \nabla_yf(y))dy\Big)\\
=&-\nabla_x\cdot\Big(\rho(t,x)\int (\nabla_x\Delta_yK_\mathrm{H}(x,y,\rho)\rho(t,y)dy\Big)\\
&+\nabla_x\cdot\Big(\rho(t,x)\int \rho(t,y)(\nabla_x\nabla_yK_\mathrm{H}(x,y,\rho), \nabla_yf(y)dy\Big),
\end{split}
\end{equation*}
where the third equality we use the fact that $\rho\nabla\log\rho=\rho$, and the last equality holds by integration by parts for {the} variable $y$. 
\end{proof}

\begin{proposition}[Particle formulation of Transport Newton's flow]
Denote $X_t\sim \rho$, $Y_t\sim \rho$ are two independent identical processes, then the particle formulation of transport Newton's flow of energy $\mathcal{E}(\rho)$ satisfies 
\begin{equation*}
\begin{split}
\frac{d}{dt}X_t=&-\mathbb{E}_{Y_t\sim\rho}\big[\nabla_x\nabla_yK_\hh(X_t, Y_t,\rho), \nabla_y\delta\mathcal{E}(\rho)(Y_t)\big].
\end{split}
\end{equation*}
In particular, if $\mathcal{E}(\rho)$ is the KL divergence functional, i.e. $\mathcal{E}(\rho)=\int \rho(x)\log\frac{\rho(x)}{e^{-f(x)}}dx$, where $e^{-f}$ is a given target distribution with $f\in C^{\infty}(M)$, then the particle formulation of transport Newton's flow satisfies 
\begin{equation*}
\frac{d}{dt}X_t=-\mathbb{E}_{Y_t\sim\rho}\big[\nabla_x\nabla_yK_\hh(X_t, Y_t,\rho), \nabla_yf(Y_t)\big]+\mathbb{E}_{Y_t\sim\rho}\big[\nabla_x\Delta_yK_\hh(X_t, Y_t,\rho)\big].
\end{equation*}
\end{proposition}
\begin{proof}
Here the particle flow follows directly from the definition of Kolmogorov forward operator.  
\end{proof}
\begin{remark}[Transport Newton kernels]\label{rmk4}
For computing transport Newton's flows, the major issue is to approximate the inverse of kernel function defined in \eqref{ker}. For example, if $\mathcal{E}(\rho)=\mathrm{D}_{\textrm{KL}}(\rho\|e^{-f})=\int \rho \log\frac{\rho}{e^{-f}} dx$, 
then 
\begin{equation*}
K_\mathrm{H}(x,y,\rho)=\Big(\nabla^2\colon(\rho \nabla^2)-\nabla\cdot(\rho\nabla^2 f\nabla)\Big)^{-1}(x,y).
\end{equation*}
For this choice of kernel function, the asymptotic convergence rate of gradient flow \eqref{Sflow} is of Newton type. In computations, we need to approximate the inverse of the second order Laplacian operator given by a kernel function $K_\hh$. This selection provides us hints for designing mean field kernel functions for accelerating Stein variational gradient; see related studies in \cite{SN, SGD}.
\end{remark}
\begin{remark}[Connecting Stein metric and transport metric]\label{rmk}
It is also worth mentioning that if $\mathcal{E}(\rho)=\frac{1}{2}\int x^2\rho(x)dx$, then the transport Hessian metric is the Wasserstein-$2$ metric. From the equivalent relation between Stein metric and transport Hessian metric, we notice that Wasserstein-$2$ metric is a particular mean field version of Stein metric. In details, if we choose the kernel function by 
\begin{equation*}
K_\mathrm{H}(x,y,\rho)=\Big(-\nabla\cdot(\rho\nabla)\Big)^{-1}(x,y),
\end{equation*}
then $\ g^\hh$ forms the Wasserstein-$2$ metric $\g$. Hence, the corresponding Stein variational derivative forms exactly the Wasserstein-$2$ gradient. 
\end{remark}
 Approximating kernel functions $K_\hh$ in above remarks are interesting future directions. 

\section{Transport Hessian Hamiltonian flows}\label{section5}
In this section, we present several dynamics associated with transport Hessian Hamiltonian flows. 

We first consider the variational formulation, a.k.a. action functional, for the proposed Hamiltonian flow.
Consider 
\begin{equation*}
\inf_{\rho}~~\Big\{\int_0^1 \frac{1}{2} \g^\hh(\partial_t\rho, \partial_t\rho)-\mathcal{F}(\rho) dt\Big\},
\end{equation*}
where the infimum is taken among all possible density paths with suitable boundary conditions on initial and terminal densities $\rho^0$, $\rho^1$, respectively. Denote $\partial_t\rho=\Delta_\rho\Phi$, then the above action functional can be formulated by
\begin{equation}\label{p1}
\begin{split}
&\inf_{\rho,\Phi}~~\Big\{\int_0^1\frac{1}{2}[\textrm{Hess}^*_{\g}\mathcal{E}(\rho)(\Phi, \Phi)]-\mathcal{F}(\rho) dt\Big\}\\
=&\inf_{\rho,\Phi}~~\Big\{\int_0^1\frac{1}{2}\Big[\int\int \nabla^2_{xy}\delta^2\mathcal{E}(\rho)(x,y)\nabla_x\Phi(t,x)\nabla\Phi(t,y)\rho(t,x)\rho(t,y)dxdy\\
&\hspace{2cm}+\int \nabla_x^2\delta\mathcal{E}(\rho)(x)(\nabla\Phi(t,x), \nabla\Phi(t,x))\rho(t,x)dx\Big]-\mathcal{F}(\rho) dt\Big\}
\end{split}
\end{equation}
where the infimum is taken among all possible density and potential functions $(\rho, \Phi)\colon [0,1]\times M]\rightarrow\mathbb{R}^2$, such that the continuity equation holds with the gradient vector drift function:  
\begin{equation*}
\partial_t\rho(t,x)+\nabla\cdot(\rho(t,x)\nabla\Phi(t,x))=0.
\end{equation*}
Here, suitable boundary conditions are given on both initial and terminal densities. 

We are ready to derive Hamiltonian flows in the Hessian density manifold. 
\begin{theorem}[Transport Hessian Hamiltonian flows]
The Hamiltonian flow in $(\mathcal{P},  \g^\hh(\rho))$ satisfies 
\begin{equation}\label{H}
\left\{\begin{split}
&\partial_t\rho+\nabla\cdot(\rho\nabla\Phi)=0\\
&\partial_t\Psi+(\nabla\Phi, \nabla\Psi)+\delta\mathcal{F}(\rho)=\frac{1}{2}\int\nabla_{yz}^2\delta^3\mathcal{E}(\rho)(x,y,z)(\nabla_y\Phi(t,y), \nabla_z\Phi(t,z))\rho(t,y)\rho(t,z)dydz\\
&\hspace{4.6cm}+\int \nabla_{xy}^2\delta^2\mathcal{E}(\rho)(x,y)(\nabla_x\Phi(t,x), \nabla_y\Phi(t,y))\rho(t,y)dy\\
&\hspace{4.6cm}+\frac{1}{2}\int \nabla^2_y\delta^2\mathcal{E}(\rho)(x,y)(\nabla_y\Phi(t,y),\nabla_y\Phi(t,y))\rho(t,y)dy\\
&\hspace{4.6cm}+\frac{1}{2}\nabla^2_x\delta\mathcal{E}(\rho)(x)(\nabla_x\Phi(t,x),\nabla_x\Phi(t,x)), 
\end{split}\right.
\end{equation}
where $\Phi$, $\Psi$ satisfy the Poisson equation 
\begin{equation}\label{PO2}
\nabla_x\cdot(\rho(t,x)\nabla_x\Psi(t,x))=\nabla_x\cdot\Big(\rho(t,x) [\nabla_x\int \nabla_y\delta^2\mathcal{E}(\rho)(x,y)\nabla_y\Phi(t,y)\rho(t,y)dy+\nabla^2_x\delta\mathcal{E}(\rho)\nabla\Phi]\Big).
\end{equation}
Here, $\delta^3$ denotes the $L^2$ third variational derivative. 
\end{theorem}
\begin{proof}
We apply the Lagrange multiplier method to solve variational problem \eqref{p1}.  
Denote $\Psi\colon [0,1]\times M\rightarrow \mathbb{R}$ as the Lagrange multiplier, then the Lagrangian functional in density space forms
\begin{equation*}
\begin{split}
\mathcal{L}(\rho,\Phi,\Psi)=&\int_0^1 \frac{1}{2}\textrm{Hess}^*_\g\mathcal{E}(\rho)(\Phi, \Phi)-\mathcal{F}(\rho)dt +\int_0^1\int \Psi\Big(\partial_t\rho+\nabla\cdot(\rho\nabla\Phi)\Big)dxdt.
\end{split}
\end{equation*}
Hence $\delta_\rho\mathcal{L}=0$, $\delta_\Phi\mathcal{L}=0$, $\delta_\Psi\mathcal{L}=0$ imply the fact that 
\begin{equation*}
\left\{\begin{split}
&\frac{1}{2}\delta_\rho \textrm{Hess}^*_\g\mathcal{E}(\rho)(\Phi, \Phi)-\delta\mathcal{F}(\rho)-\partial_t\Psi-(\nabla\Phi, \nabla\Psi)=0\\
&\nabla_x\cdot(\rho(t,x) \nabla_x\int \nabla_y\delta^2\mathcal{E}(\rho)(x,y)\nabla_y\Phi(t,y)\rho(t,y)dy)+\nabla_x\cdot(\rho\nabla_x^2\delta\mathcal{E}(\rho)\nabla_x\Phi)=\nabla\cdot(\rho\nabla\Psi)\\
&\partial_t\rho+\nabla\cdot(\rho\nabla\Phi)=0,
\end{split}\right.
\end{equation*}
which finishes the proof.
\end{proof}
In above derivations, we notice that $\Psi$ is the momentum variable in $\g^{\H}$ and $\Phi$ is the momentum variable in $\g$.
Here the Hamiltonian in $(\mathcal{P}, g^{\H})$ satisfies 
\begin{equation*}
\mathcal{H}(\rho,\Psi)=\frac{1}{2}\textrm{Hess}^*_{\g}\mathcal{E}(\rho)(\Phi, \Phi)+\mathcal{F}(\rho),
\end{equation*}
where $\Phi,\Psi$ are associated with the Poisson equation \eqref{PO2}. We notice the Hamiltonian flow \eqref{H} is equivalent to the following formulation  
\begin{equation*}
\partial_{tt}\rho+\bf{\Gamma}^H(\rho)(\partial_t\rho, \partial_t\rho)=-\textrm{grad}_{ \g^\hh}\mathcal{F}(\rho).
\end{equation*}
where $\Gamma^H(\rho)$ is the Christoffel symbol in transport Hessian metric. We omit the detailed formulations fo $\Gamma^H$; see details in \cite{LiG}. In particular, when $\mathcal{F}(\rho)=0$, the above Hamiltonian flow {leads to} the geodesic equation in $(\mathcal{P},  \g^\hh)$.  
\subsection{Examples}
We next present several examples of transport Hessian Hamiltonian flows. 
\begin{example}[Linear energy]
Consider $\mathcal{E}(\rho)=\int E(x)\rho(x)dx$. Then the transport Hessian Hamiltonian flow forms 
\begin{equation*}
\left\{\begin{split}
&\partial_t\rho+\nabla\cdot(\rho\nabla\Phi)=0\\
&\partial_t\Psi+(\nabla\Phi, \nabla\Psi)-\frac{1}{2}\nabla^2E(x)(\nabla\Phi,\nabla\Phi)+\delta\mathcal{F}(\rho)=0\\
&\nabla\cdot(\rho\nabla\Psi)=\nabla\cdot(\rho\nabla^2E(x)\nabla\Phi).
\end{split}\right.
\end{equation*}
Again, if $E(x)=\frac{1}{2}\int x^2\rho(x)dx$, then the Hamiltonian flow satisfies 
\begin{equation*}
\left\{\begin{split}
&\partial_t\rho+\nabla\cdot(\rho\nabla\Phi)=0\\
&\partial_t\Phi+\frac{1}{2}(\nabla\Phi, \nabla\Phi)+\delta\mathcal{F}(\rho)=0.
\end{split}\right.
\end{equation*}
The above equation system is known as the {Wasserstein Hamiltonian flow} \cite{LC, LiG}. It is a particular format for compressible Euler equations.
In particular, if $\mathcal{F}(\rho)=0$, it is the {Wasserstein geodesics equation} \cite{Villani2009_optimal}. 
\end{example}

\begin{example}[Interaction energy]
Consider $\mathcal{E}(\rho)=\frac{1}{2}\int\int W(x,y)\rho(x)\rho(y)dxdy$. Then the transport Hessian Hamiltonian flow satisfies
\begin{equation*}
\left\{\begin{split}
&\partial_t\rho+\nabla\cdot(\rho\nabla\Phi)=0\\
&\partial_t\Psi+(\nabla\Phi, \nabla\Psi)+\delta\mathcal{F}(\rho)=\int \nabla_{xy}^2W(x,y)(\nabla\Phi(t,x),\nabla\Phi(t,y))\rho(t,y)dy\\
&\hspace{4.6cm}+\frac{1}{2}\int \nabla^2_yW(x,y)(\nabla_y\Phi(t,y),\nabla_y\Phi(t,y)) \rho(t,y)dy\\
&\hspace{4.6cm}+\frac{1}{2}\int \nabla^2_xW(x,y)(\nabla_x\Phi(t,x), \nabla_x\Phi(t,x))\rho(t,y)dy\\
&\nabla\cdot(\rho\nabla\Psi)=\nabla_x\cdot\Big(\rho(t,x)\int [\nabla^2_{xy}W(x,y)\nabla_y\Phi(t,y)+\nabla^2_xW(x,y)\nabla\Phi(t,x)]\rho(t,y)dy\Big).
\end{split}\right.
\end{equation*}
\end{example}

\begin{example}[Entropy]
Consider $\mathcal{E}(\rho)=\int f(\rho)(x)dx$. Then the transport Hessian Hamiltonian flow satisfies
\begin{equation*}
\left\{\begin{split}
&\partial_t\rho+\nabla\cdot(\rho\nabla\Phi)=0\\
&\partial_t\Psi+(\nabla\Phi, \nabla\Psi)-\frac{1}{2}\|\nabla^2\Phi\|^2p'(\rho)-\frac{1}{2}|\Delta\Phi|^2p'_2(\rho)+\delta\mathcal{F}(\rho)=0\\
&-\nabla\cdot(\rho\nabla\Psi)=\nabla^2\colon(p(\rho)\nabla^2\Phi)+\Delta(p_2(\rho)\Delta \Phi).
\end{split}\right.
\end{equation*}
\end{example}
We next present several examples for transport Hessian Hamiltonian flows of entropies. 
\begin{example}[Quadratic entropy]
Consider $\mathcal{E}(\rho)=\frac{1}{2}\int \rho(x)^2dx$. Then the transport Hessian Hamiltonian flow satisfies
\begin{equation*}
\left\{\begin{split}
&\partial_t\rho+\nabla\cdot(\rho\nabla\Phi)=0\\
&\partial_t\Psi+(\nabla\Phi, \nabla\Psi)-\frac{1}{2}\|\nabla^2\Phi\|^2\rho-\frac{1}{2}|\Delta\Phi|^2\rho+\delta\mathcal{F}(\rho)=0\\
&-\nabla\cdot(\rho\nabla\Psi)=\frac{1}{2}\nabla^2\colon(\rho\nabla^2\Phi)+\frac{1}{2}\Delta(\rho\Delta \Phi).
\end{split}\right.
\end{equation*}
\end{example}
\begin{example}[{Boltzmann--Shannon entropy}]\label{BS}
Consider $\mathcal{E}(\rho)=\int\rho(x)\log\rho(x)dx$. Then the transport Hessian Hamiltonian flow satisfies
\begin{equation}\label{GBS}
\left\{\begin{split}
&\partial_t\rho+\nabla\cdot(\rho\nabla\Phi)=0\\
&\partial_t\Psi+(\nabla\Phi, \nabla\Psi)-\frac{1}{2}\|\nabla^2\Phi\|^2+\delta\mathcal{F}(\rho)=0\\
&-\nabla\cdot(\rho\nabla\Psi)=\nabla^2\colon(\rho\nabla^2\Phi).
\end{split}\right.
\end{equation}
\end{example}
\subsection{Connections with Shallow water equation}
We next demonstrate that equation \eqref{GBS} connects with the Shallow water equation. If $M=\mathbb{T}^1$ and
\begin{equation*}
\mathcal{E}(\rho)=\int\rho\log\rho dx,\qquad\mathcal{F}(\rho)=-\frac{1}{2}\int \rho^2 dx,
\end{equation*}
then equation \eqref{GBS} is known as the Shallow water equation in fluid dynamics; see \cite{Modin1} and many references therein. In other words, in one dimensional sample space, denote $v(t,x)=\nabla\Phi(t,x)$, then equation \eqref{GBS} forms 
\begin{equation*}
\left\{\begin{split}
&\partial_t\rho+\nabla\cdot(\rho v)=0\\
&\partial_t\Psi+(v, \nabla\Psi)-\frac{1}{2}\|\nabla v\|^2-\rho=0\\
&-\rho\nabla\Psi=\nabla\cdot(\rho\nabla v).
\end{split}\right.
\end{equation*}
The above equation system is the minimizer of variational problem 
\begin{equation}\label{p1}
\inf_{\rho, v=\nabla\Phi} \Big\{\int_0^1 \int \Big(\|\nabla v(t,x)\|^2\rho(t,x)+\frac{1}{2}\rho(t,x)^2 \Big)dxdt\colon \partial_t\rho(t,x)+\nabla\cdot(\rho(t,x) v(t,x))=0\Big\},
\end{equation}
where the infimum is taken among continuity equations with {\em gradient vector fields} and suitable boundary conditions on initial and terminal densities. Here $\Psi$ is the Lagrange multiplier of the continuity equation. 
\begin{remark}
We remark that the Hamiltonian flows in Hessian density manifold of {Boltzmann--Shanon entropy} coincides with the ones using Hessian operators in diffeomorphism space \cite{SW} or semi-invariant metric \cite{Modin1} in one dimensional sample space. Other than one dimensional space, these metrics and related variational structures are different. In other words, consider the variational problem 
\begin{equation}\label{p2}
\inf_{\rho, v}\Big\{\int_0^1\int \Big(\|\nabla v(t,x)\|^2\rho(t,x)+\frac{\rho^2}{2}\Big) dxdt \colon \partial_t\rho(t,x)+\nabla\cdot(\rho(t,x)v(t,x))=0\Big\},
\end{equation}
where the infimum is taken among continuity equations with all vector fields and suitable initial and terminal densities. When $\textrm{dim}(M)=1$, variational problems  \eqref{p1} and \eqref{p2} are equivalent. If $\textrm{dim}(M)\neq 1$, they are with different formulations. 
\end{remark}
To summarize, we connect the Shallow water equation with {Boltzmann--Shannon entropy} and transport Hessian metrics.

\section{Finite dimensional Transport Hessian metric}\label{section7}
In this section, we demonstrate transport Hessian metrics in finite dimensional probability models. { Here the probability models are commonly used in applications, based on AI and classical computational physics methods, such as finite element or finite volume methods. In other words, the probability model represents the subset of probability measure space, using which one can numerically approximate the dynamical behaviors of mathematical physics equations}. Shortly, we pull back the transport Hessian metric into a finite dimensional parameter space. Several examples of distance functions are provided, including one dimensional probability models and Gaussian families.

For the simplicity of presentation, {we focus on the transport Hessian metric  \eqref{SE} of negative Shannon--Boltzmann entropy, i.e. $\int \rho\log\rho dx$}. Consider a statistical model defined by a triplet $(\Theta,  M, \rho)$. For the simplicity of presentation, we assume $\Theta\subset\mathbb{R}^d$ and $\rho\colon \Theta\rightarrow \mathcal{P}( M)$ is a parametrization function. In this case, $\rho(\Theta)\subset \mathcal{P}( M)$. We assume that the parameterization map $\rho$ is locally injective and smooth. Given $\rho(\theta, \cdot)\in \rho(\Theta)$, denote the tangent space of probability density space by
\begin{equation*}
T_{\rho_\theta}\rho(\Theta)=\big\{\dot\rho_\theta=(\nabla_\theta\rho_\theta, \dot\theta)\in C^{\infty}( M)\colon \dot\theta\in T_{\theta}\Theta\big\}.
\end{equation*}
We define a Riemannian metric $G_\hh(\theta)=(G_\hh(\theta)_{ij})_{1\leq i,j\leq d}\in \mathbb{R}^{d\times d}$ on $\Theta$ as the pull-back of $ \g^\hh$ on $\mathcal{P}$, i.e. 
\begin{equation}\label{pullback}
\dot\theta^{\ts}G_\hh(\theta)\dot\theta= \g^\hh_{\rho_\theta}(\dot\rho_\theta, \dot\rho_\theta). 
\end{equation}
Here $G_\hh$ is a semi-positive matrix function and $(\Theta, G_\hh)$ is named the statistical manifold. Denote by $\nabla$ the Euclidean derivative w.r.t $x$ and $\partial_\theta$ the Euclidean derivative w.r.t. $\theta$. We are now ready to present the transport Hessian metric.  
\begin{definition}[Transport Hessian information matrix]
The transport Hessian metric in statistical manifold $(\Theta, G_\hh)$ satisfies 
\begin{equation*}
G_{H}(\theta)_{ij}=\int \mathrm{tr}\big(\nabla^2\Phi_i(x;\theta), \nabla^2\Phi_j(x;\theta)\big) \rho(x;\theta)dx,
\end{equation*}
where
\begin{equation*}
-\nabla\cdot(\rho(x;\theta) \nabla \Phi_k(x;\theta))=\nabla_{\theta_k}\rho(x;\theta),\qquad \textrm{$k=i,j$.}
\end{equation*}
\end{definition}
\begin{remark}
Following studies in \cite{LiG5}, we call this finite dimensional metric the {\em transport Hessian information matrix}. This is the other generalization of information matrix studied in \cite{IG, IG1, IG2}.
We leave the systematic studies of transport information matrix in both statistics and {AI} inference problems for the future work. 
\end{remark}
\begin{proof}
The definition follows from the definition of pull-back operator. Notice that 
\begin{equation*}
\begin{split}
G_{H}(\theta)_{ij}=& \g^\hh(\rho)(\nabla_{\theta_i}\rho, \nabla_{\theta_j}\rho)\\
=&\int \Big(\nabla_{\theta_i} \rho(x;\theta), (\nabla\cdot \rho \nabla)^{-1}\Big(\nabla^2\colon \rho \nabla^2\Big)(\nabla\cdot \rho\nabla)^{-1}\nabla_{\theta_j}\rho(x;\theta)\Big) dx.
\end{split}
\end{equation*}
By denoting {$-\nabla\cdot(\rho \nabla \Phi_k)=\nabla_{\theta_k}\rho$, $k=i,j$}, we finish the proof.
\end{proof}

We next present several closed formulas of finite dimensional transport Hessian information metrics. 
\begin{example}[{Transport Hessian metric in one dimensional sample space}]
If $M=\mathbb{T}^1$, then
\begin{equation*}
\nabla\Phi_k(x;\theta)=-\frac{1}{\rho(x;\theta)}\partial_{\theta_k}F(x;\theta),
\end{equation*}
where $k=1,\cdots, d$ and $F$ is the cumulative distribution function of $\rho(x;\theta)$ with $F(x;\theta)=\int^x_0 \rho(y;\theta)dy$. In this case, we have
\begin{equation*}
\begin{split}
G_\hh(\theta)_{ij}=&\int \nabla^2\Phi_i \nabla^2\Phi_j \rho dx\\
=&\int \nabla (\frac{\partial_{\theta_i}F}{\rho})\cdot \nabla(\frac{\partial_{\theta_j}F}{\rho})\rho dx\\
=&\int (\frac{\nabla \partial_{\theta_i}F}{\rho}+\partial_{\theta_i}F\nabla \frac{1}{\rho})(\frac{\nabla \partial_{\theta_j}F}{\rho}+\partial_{\theta_j}F\nabla \frac{1}{\rho})\rho dx\\
=&\int (\frac{\partial_{\theta_i}\rho}{\rho}+\partial_{\theta_i}F\nabla \frac{1}{\rho})(\frac{\partial_{\theta_j}\rho}{\rho}+\partial_{\theta_j}F\nabla \frac{1}{\rho})\rho dx\\
=&\int \frac{\partial_{\theta_i}\rho \partial_{\theta_j}\rho}{\rho}+\partial_{\theta_i}F\nabla \frac{1}{\rho}\partial_{\theta_j}\rho  +\partial_{\theta_j}F\nabla \frac{1}{\rho}\partial_{\theta_i}\rho
+\partial_{\theta_i}F\partial_{\theta_j}F(\nabla\frac{1}{\rho})^2\rho dx\\
=&\quad\int \big(\rho \partial_{\theta_i}\log\rho \partial_{\theta_j}\log\rho\big) dx\\
&+\int\big(-\partial_{\theta_i}F\nabla\log\rho\partial_{\theta_j}\log\rho-\partial_{\theta_j}F\nabla\log\rho\partial_{\theta_i}\log\rho
+\frac{\partial_{\theta_i}F\partial_{\theta_j}F}{\rho}(\nabla\log\rho)^2\big)dx,
\end{split}
\end{equation*}
where we use the fact that $\nabla\frac{1}{\rho}=-\frac{\nabla\rho}{\rho^2}=-\frac{1}{\rho}\nabla\log\rho$, $\nabla\log\rho=\frac{\nabla \rho}{\rho}$ and $\partial_\theta\log\rho=\frac{\partial_\theta\rho}{\rho}$ in the last equality. 
We notice that the transport Hessian metric is a modification of Fisher--Rao metric with the transport Levi-Civita connection; see details in \cite{LiG}. 
\end{example}
\begin{remark}
We compare all information matrices in one dimensional sample space as follows:
\begin{equation*}
\begin{split}
G_F(\theta)_{ij}=&\int \frac{\partial_{\theta_i}\rho(x;\theta)\partial_{\theta_j}\rho(x;\theta)}{\rho(x;\theta)}dx\hspace{3.3cm}\textrm{Fisher information matrix}\\
G_W(\theta)_{ij}=&\int \frac{\partial_{\theta_i}F(x;\theta)\partial_{\theta_j}F(x;\theta)}{\rho(x;\theta)}dx\hspace{3cm}\textrm{Wasserstein information matrix}   \\
G_{\mathrm{H}}(\theta)_{ij}=&\int \nabla (\frac{\partial_{\theta_i}F(x; \theta)}{\rho(x; \theta)})\cdot \nabla(\frac{\partial_{\theta_j}F(x;\theta)}{\rho(x;\theta)})\rho(x;\theta) dx\quad\textrm{Transport Hessian information matrix}
\end{split}
\end{equation*}
\end{remark}

\begin{example}[{Transport Hessian distance in Gaussian family}]\label{Gaussian}
Consider a Gaussian family by
\begin{equation}\label{GF}
p(x;\theta)=\frac{1}{\sqrt{2\pi\Sigma}}e^{-\frac{1}{2}(x-m)^{\ts}\Sigma^{-1}(x-m)},
\end{equation}
where $x\in\mathbb{R}^d$, $\theta=(m,\Sigma)$, $m\in\mathbb{R}^{d}$ is a mean value vector, and $\Sigma\in \mathbb{R}^{d\times d}$ is a positive definite covariance matrix. In this case, 
the continuity equation with the gradient drift has a closed form solution. 
Denote $\dot\theta=(\dot m, \dot \Sigma)\in T_\theta\Theta$. We can check that if \begin{equation*}
-\nabla\cdot(\rho \nabla\Phi)=(\nabla_\theta\rho, \dot\theta),
\end{equation*}
then there exists a symmetric matrix $S\in\mathbb{R}^{d\times d}$ and $b\in \mathbb{R}^d$, such that 
\begin{equation*}
\Phi(x)=\frac{1}{2}x^{\ts}Sx+b^{\ts}x,
\end{equation*}
and 
\begin{equation*}
\dot \Sigma=\Sigma S+S\Sigma,\qquad \dot m=b. 
\end{equation*}
Hence the transport Hessian metric in Gaussian family satisfies
\begin{equation*}
G_\hh(\Sigma)((\dot m_1, \dot \Sigma_1), (\dot m_2, \dot \Sigma_2))= \mathrm{tr}(S_1S_2),
\end{equation*}
where
\begin{equation*}
\dot \Sigma_1 =\Sigma S_1+S_1\Sigma,\quad \dot \Sigma_2 =\Sigma S_2+S_2\Sigma.
\end{equation*}
We notice that the transport Hessian metric is degenerate in the direction of mean values. Besides, we consider a Gaussian family with zero mean in one dimensional space. In this case, 
\begin{equation*}
G_\hh(\Sigma)( \dot \Sigma_1, \dot \Sigma_2)= S_1S_2=\frac{\dot\Sigma_1\dot \Sigma_2}{4\Sigma^2}.
\end{equation*}
Hence the distance function has a closed form solution. Notice 
\begin{equation*}
\begin{split}
\mathrm{Dist}_{\mathrm{H}}(\Sigma_0, \Sigma_1)^2=&\inf_{\Sigma\colon [0,1]\rightarrow \mathbb{R}_+} \Big\{\int_0^1 G_\hh(\frac{d}{dt} \Sigma(t), \frac{d}{dt}\Sigma(t))dt\colon ~\textrm{$\Sigma_0$, $\Sigma_1$ fixed}\Big\}\\
=&\inf_{\Sigma\colon [0,1]\rightarrow \mathbb{R}_+} \Big\{\int_0^1 \frac{(\frac{d}{dt}\Sigma(t))^2}{4\Sigma^2}dt\colon ~\textrm{$\Sigma_0$, $\Sigma_1$ fixed}\Big\}\\
=&\inf_{\Sigma\colon [0,1]\rightarrow \mathbb{R}_+} \Big\{\int_0^1 \frac{1}{4}(\frac{d}{dt}\log{\Sigma(t)})^2dt~\colon ~\textrm{$\Sigma_0$, $\Sigma_1$ fixed}\Big\},
\end{split}
\end{equation*}
where we use the fact that $\frac{d}{dt}\log{\Sigma(t)}=\frac{\frac{d}{dt}\Sigma(t)}{{\Sigma(t)}}$. Hence the geodesics satisfies $\frac{d^2}{dt^2}\log{\Sigma(t)}=0$, and $\log{\Sigma(t)}=(1-t)\log{\Sigma_0}+t\log{\Sigma_1}$. Hence we derive the closed form solution for the transport Hessian metric: 
\begin{equation*}
\mathrm{Dist}_{\mathrm{H}}(\Sigma_0, \Sigma_1)=\frac{1}{2}\|\log\Sigma_0-\log\Sigma_1\|. 
\end{equation*}
In this case, we observe that the transport Hessian metric coincides with the Fisher--Rao metric for covariance matrices in Gaussian families.  
\end{example}

\begin{example}[Generalized transport Hessian distance in Gaussian family]
In this example, we consider a finite dimensional transport Hessian metric for the energy $\mathcal{E}(\rho)=\int(\rho\log\rho+\frac{x^2}{2}\rho)dx$. In this case, the metric $ \g^\hh$ forms 
\begin{equation*}
 \g^\hh(\rho)(\sigma_1, \sigma_2)=\int \Big\{\mathrm{tr}(\nabla^2\Phi_1(x),\nabla^2\Phi_2(x))+(\nabla\Phi_1(x), \nabla\Phi_2(x))\Big\}\rho(x) dx.
\end{equation*}
where $\sigma_i=-\nabla\cdot(\rho\nabla\Phi_i)$, $i=1,2$. Similarly as in Example \ref{Gaussian}, we consider the Gaussian family \eqref{GF} in one dimensional sample space. Then 
\begin{equation*}
G_\hh(m,\Sigma)( (\dot m_1, \dot \Sigma_1), (\dot m_2, \dot \Sigma_2))= S_1S_2+S_1\Sigma S_2+b_1b_2=\frac{\dot\Sigma_1\dot \Sigma_2}{4\Sigma^2}+\frac{\dot\Sigma_1\dot\Sigma_2}{4\Sigma}+\dot m_1\dot m_2.
\end{equation*}
In this case, the transport Hessian distance function has a closed form solution. Similarly, 
\begin{equation*}
\begin{split}
&\mathrm{Dist}_{\mathrm{H}}((m_0,\Sigma_0), (m_1,\Sigma_1))^2\\
=&\inf_{\Sigma\colon [0,1]\rightarrow \mathbb{R}_+} \Big\{\int_0^1 G_\hh(\frac{d}{dt} \Sigma(t), \frac{d}{dt}\Sigma(t))dt\colon ~\textrm{$(m_0, \Sigma_0)$, $(m_1, \Sigma_1)$ fixed}\Big\}\\
=&\inf_{\Sigma\colon [0,1]\rightarrow \mathbb{R}_+} \Big\{\int_0^1 (\frac{d}{dt}\Sigma(t))^2(\frac{1}{4\Sigma(t)^2}+\frac{1}{4\Sigma(t)})+(\frac{d}{dt}m(t))^2dt\colon ~\textrm{$(m_0, \Sigma_0)$, $(m_1, \Sigma_1)$ fixed}\Big\}\\
=&\inf_{\Sigma\colon [0,1]\rightarrow \mathbb{R}_+} \Big\{\int_0^1 \|\frac{d}{dt}\big(\sqrt{\Sigma(t)+1}-\mathrm{tanh}^{-1}(\sqrt{\Sigma(t)+1})\big)\|^2+\|\frac{d}{dt}m(t)\|^2dt~\colon ~\textrm{$(m_0, \Sigma_0)$, $(m_1,\Sigma_1)$ fixed}\Big\},
\end{split}
\end{equation*}
where we use the fact that $\frac{d}{dt}\big(\sqrt{\Sigma(t)+1}-\mathrm{tanh}^{-1}(\sqrt{\Sigma(t)+1})\big)=\frac{1}{2}\sqrt{\frac{1}{\Sigma(t)^2}+\frac{1}{\Sigma(t)}}\frac{d}{dt}\Sigma(t)$. 
Hence the geodesics forms 
\begin{equation*}
\left\{\begin{split}
&\frac{d^2}{dt^2}\big(\sqrt{\Sigma(t)+1}-\mathrm{tanh}^{-1}(\sqrt{\Sigma(t)+1})\big)=0\\
&\frac{d^2}{dt^2}m(t)=0.
\end{split}\right.
\end{equation*}
Hence the geodesics satisfies 
\begin{equation*}
\left\{\begin{split}
&\sqrt{\Sigma(t)+1}-\mathrm{tanh}^{-1}(\sqrt{\Sigma(t)+1})=(1-t)(\sqrt{\Sigma_0+1}-\mathrm{tanh}^{-1}(\sqrt{\Sigma_0+1}))\\
&\hspace{6cm}+t (\sqrt{\Sigma_1+1}-\mathrm{tanh}^{-1}(\sqrt{\Sigma_1+1})) \\
&m(t)=(1-t)m_0+tm_1
\end{split}\right.
\end{equation*}
Here we derive the closed form solution of the transport Hessian metric as follows:
\begin{equation*}
\begin{split}
&\mathrm{Dist}_{\mathrm{H}}((m_0,\Sigma_0), (m_1,\Sigma_1))^2\\
=&{\|m_0-m_1\|^2+\|\sqrt{\Sigma_0+1}-\sqrt{\Sigma_1+1}-(\mathrm{tanh}^{-1}(\sqrt{\Sigma_0+1})-\mathrm{tanh}^{-1}(\sqrt{\Sigma_1+1}))\|^2}.
\end{split}
\end{equation*}
\end{example}
There are many other interesting closed form solutions of {transport Hessian metrics and distances}. We shall derive them in the future work. 
\section{Discussions}
In this paper, we study the Hessian metric of a functional in Wasserstein-$2$ space. We name the density space with this transport Hessian metric the Hessian density manifold. 
We discover several connections between dynamics in Hessian density manifold and mathematical physics equations, including Shallow water equations and heat equations. In particular, we demonstrate that a transport Hessian metric induces a mean-field kernel for the Stein metric, following which the Stein variational derivative {leads to} the transport Newton's direction.

Following the transport Hessian metric, there are several future directions among Hessian metrics, scientific computing methods, and machine learning sampling algorithms. Firstly, we notice that the optimal transport metrics belong to a particular class of Hessian metric in density manifold. For these transport Hessian metrics, are there any other formulations, such as Monge problems, Monge-Amper{\'e} equations, or Kantorovich dualities as in optimal transport? {We will study this direction following \cite{AGS}}. Secondly, can we formulate {\em transport divergence functionals} based on transport Hessian metrics? We will conduct this direction following the study of information geometry; Thirdly,  the proposed structure can involve several variational formulations for mathematical physics equations. Can this method be useful for designing numerical algorithms towards Shallow water equations? Lastly, the proposed metric provides a hint for selecting the kernel to accelerate Stein variational sampling algorithms. We shall propose numerical methods to approximate kernels in the proposed Stein--transport Newton's flow. This could be useful in designing mean-field Markov-chain-Monte-Carlo methods.

\noindent\textbf{Data Availability Statement}. Data sharing is not applicable to this article as no new data were created or analyzed in this study.

\end{document}